\documentclass[conference]{IEEEtran}  

\usepackage{etoolbox}
\newtoggle{ext}     
\toggletrue{ext}    

\usepackage{times}

\usepackage[numbers]{natbib}
\usepackage{multicol}
\usepackage[bookmarks=true]{hyperref}

\IEEEoverridecommandlockouts                              

\usepackage{etoolbox}
\usepackage{tabularx}
\usepackage{graphics} 
\usepackage{epsfig} 
\usepackage{mathptmx} 
\usepackage{times} 
\usepackage{amsmath} 
\usepackage{amssymb}  
\usepackage{amsthm}
\usepackage{bm}  
\usepackage{newtxtext, newtxmath}
\usepackage[usenames, dvipsnames]{color}
\usepackage{caption}
\usepackage{subcaption}

\usepackage{algorithm}
\usepackage{algpseudocode}

\usepackage{eufrak}
\usepackage{stackengine}
\usepackage{multirow}

\newtheoremstyle{nospace}
{2pt}   				
{2pt}   				
{\itshape}  			
{} 		  		    	
{\bfseries} 			
{.}         			
{5pt plus 1pt minus 1pt}
{}          			

\theoremstyle{nospace} \newtheorem{theorem}{Theorem}
\theoremstyle{nospace} 
\theoremstyle{nospace} \newtheorem{lemma}[theorem]{Lemma}
\theoremstyle{nospace} \newtheorem{remark}{Remark}
\theoremstyle{nospace} 
\theoremstyle{nospace} 
\theoremstyle{nospace} \newtheorem{corollary}{Corollary}
\theoremstyle{nospace} 
\theoremstyle{nospace} 

\algnewcommand{\IfThen}[2]{
  \State \algorithmicif\ #1\ \algorithmicthen\ #2}
  
\makeatletter
\let\OldStatex\Statex
\renewcommand{\Statex}[1][3]{%
  \setlength\@tempdima{\algorithmicindent}%
  \OldStatex\hskip\dimexpr#1\@tempdima\relax}
\makeatother

\renewcommand{\baselinestretch}{0.93}

\title{Trajectory Optimization on Manifolds: A Theoretically-Guaranteed Embedded Sequential Convex Programming Approach}

\author{Riccardo Bonalli, Andrew Bylard, Abhishek Cauligi, Thomas Lew, Marco Pavone 
\thanks{R. Bonalli, A. Bylard, A. Cauligi, and M. Pavone are with the Department of Aeronautics and Astronautics, Stanford University, Stanford, CA 94305. \{\tt rbonalli, bylard, acauligi, pavone\} {\tt@stanford.edu}.}
\thanks{T. Lew is with the Institute for Dynamic Systems and Control, ETH Z{\"u}rich, Z{\"u}rich, Switzerland.
{\tt lewt@student.ethz.ch}.}
\thanks{This work was supported in part by NASA under the Space Technology Research Program, NASA Space Technology Research Fellowship Grants NNX16AM78H and NNX15AP67H, Early Career Faculty Grant NNX12AQ43G and Early Stage Innovations Grant NNX16AD19G, and by KACST.
T. Lew is partially supported by the Master's Thesis Grant of the Zeno Karl Schindler Foundation.}
}

\begin{document}

\maketitle
\thispagestyle{empty}
\pagestyle{empty}


\begin{abstract}
Sequential Convex Programming (SCP) has recently gained popularity as a tool for trajectory optimization due to its sound theoretical properties and practical performance. Yet, most SCP-based methods for trajectory optimization are restricted to Euclidean settings, which precludes their application to problem instances where one must reason about manifold-type constraints (that is, constraints, such as loop closure, which restrict the motion of a system to a subset of the ambient space). The aim of this paper is to fill this gap by extending SCP-based trajectory optimization methods to a manifold setting. The key insight is to leverage {\em geometric embeddings} to lift a manifold-constrained trajectory optimization problem into an equivalent problem defined over a space enjoying a Euclidean structure. This insight allows one to extend existing SCP methods to a manifold setting in a fairly natural way. In particular, we present a SCP algorithm for manifold problems with refined theoretical guarantees that resemble those derived for the Euclidean setting, and demonstrate its practical performance via numerical experiments.
\end{abstract}

\IEEEpeerreviewmaketitle

\section{Introduction} 
\label{sec:introduction}

Trajectory optimization is a key problem in robotics, and it has thus been studied extensively through a variety of mathematical frameworks. Examples include sampling-based motion planning techniques  \cite{JansonSchmerlingEtAl2015, KaramanFrazzoli2011,KavrakiSvestkaEtAl1996,LaValle2006,LaValleKuffner2000}, variational approaches such as CHOMP and STOMP  \cite{RatliffZuckerEtAl2009,KalakrishnanChittaEtAl2011}, sum-of-squares methods  \cite{MajumdarTedrake2017,SinghChenEtAl2018}, and sequential convex programming (SCP) techniques such as TrajOpt and GuSTO  \cite{BonalliCauligiEtAl2019,LiuLu2014,MaoSzmukEtAl2016,SchulmanDuanEtAl2014}. Most of these methods, however, are restricted to Euclidean settings, which precludes their application (at least directly) to problem instances where one needs to reason about {\em manifold-type constraints}. For example, such constraints arise when the motion of a robotic system is forced to evolve on subsets of the ambient space (e.g., due to the presence of closed kinematic chains giving rise to loop closure constraints \cite{JailletPorta2016}), which are mathematically modeled as manifolds. Systems having such constraints include quadrotors \cite{AugugliaroSchoelligEtAl2012,MellingerMichaelKumar2012}, robots with camera orientation constraints \cite{ShenMulgaonkarEtAl2013,WattersonLiuEtAl2018}, manipulator systems \cite{KhorasaniKokotovic1985,MurrayLiEtAl1994} and robotic spacecraft \cite{MisraBai2017,Virgili-llopZagarisEtAl2019}, to name a few. For such systems, trajectory optimization methods must ensure that the computed trajectories lie on the relevant manifolds, preventing the planning of infeasible motions. However, this is in general challenging, as  manifold-type constraints are often defined only locally (i.e., through local charts) or as implicit constraints (i.e., constraints of the type $E(x)=0$ where $E$ is a submersion and $x$ is the state vector). As a pedagogical example, consider a two-joint manipulator. Its motion is forced to evolve on the torus $\mathbb{T}^2$ (see Figure \ref{figTorus}), which represents a two-dimensional submanifold of $\mathbb{R}^3$. Specifically, each joint variable ($\theta_i$, $i=1,2$) evolves on the unitary circle $S^1$, and thus the combined evolution is on the Cartesian product $S^1\times S^1$, which is diffeomorphic to $\mathbb{T}^2$ (as characterized by implicit and nonlinear equality constraints).

\begin{figure}[t!]
\centering

\vspace{10pt}

\includegraphics[width=0.95\columnwidth]{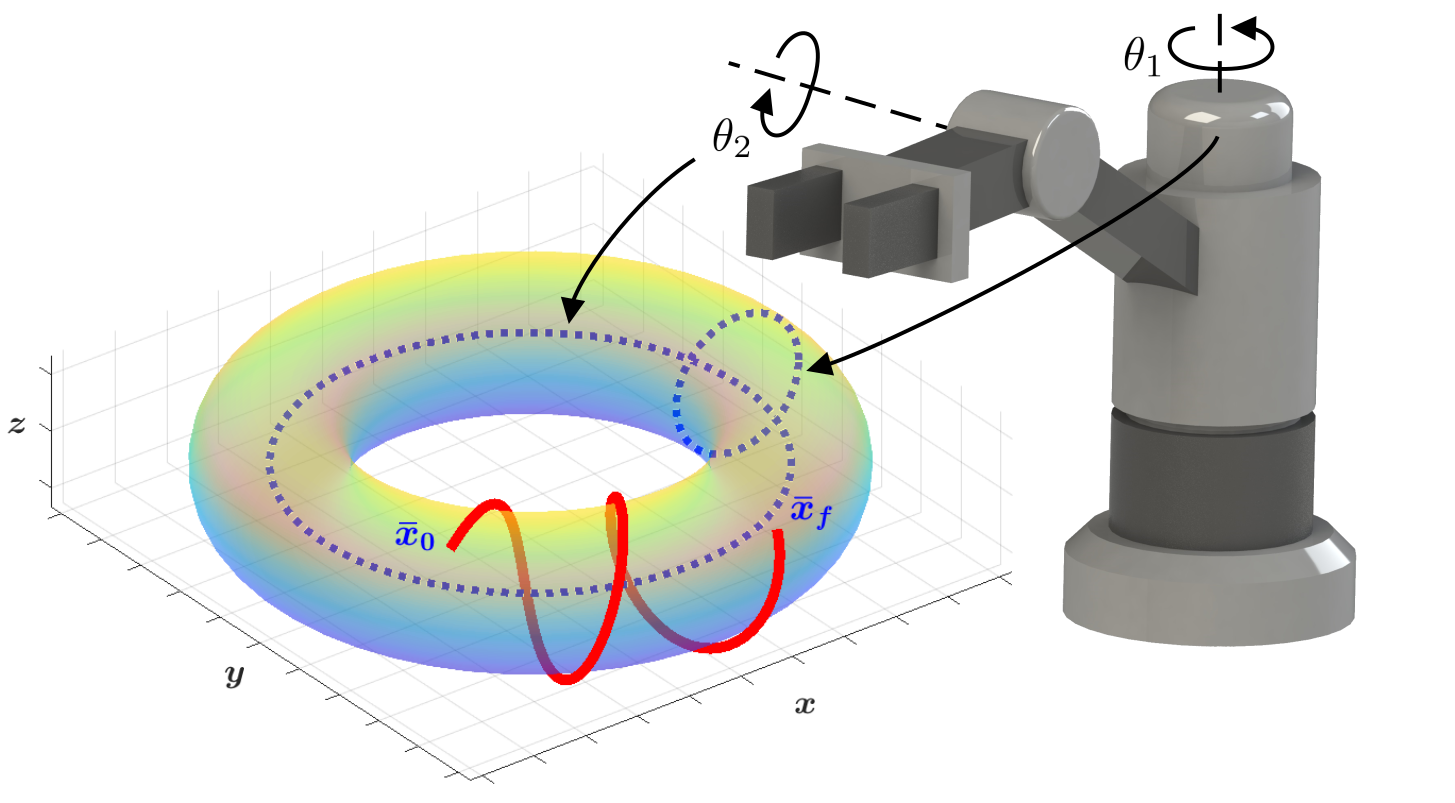}

\caption{The torus $\mathbb{T}^2$ is an embedded submanifold of $\mathbb{R}^3$. The red trajectory, starting from $\bar{x}_0$ and finishing at $\bar{x}_f$, is characterized by dynamics with torus-type constraints and can be seen as an embedded curve in $\mathbb{R}^3$ (Section \ref{subsec:embedding}).} \label{figTorus}

\vspace{-20pt}
\label{fig:torus}
\end{figure}

Despite the ubiquitous presence of manifold constraints in robotic applications, the set of trajectory optimization tools that handle such constraints is relatively limited.  The most na\"ive technique consists of simply removing, without principled justification, all manifold-type constraints, and then solving a relaxed version of the original problem in the resulting Euclidean space. Since this approach cannot guarantee trajectory feasibility, one needs to resort to post-processing before trajectory execution, often using a heuristic correction step which may be unsuccessful. This has prompted the design of optimization approaches that explicitly account for the presence of manifold constraints, including sampling-based techniques leveraging local chart analysis, \cite{SuhUmEtAl2011,JailletPorta2016}, methods employing global chart-gluing procedures \cite{WattersonLiuEtAl2018}, and methods exploiting properties of specific types of manifolds (in particular, Lie groups), such as invariant metrics \cite{WattersonSmithEtAl2016} and projection operators \cite{SacconAguiarEtAl2012,SacconHauserEtAl2013}. These methods, while directly accounting for the presence of manifold constraints, do not  in general enjoy theoretical guarantees, and they only consider a subset of the typical constraints arising in robotic applications (e.g., control or goal region constraints are  generally not addressed).

Building on the recent success of SCP-based techniques for trajectory optimization, the aim of this paper is to provide an SCP-based framework for trajectory optimization on manifolds that (1) enjoys theoretical guarantees in terms of convergence, (2) is general, in that it accounts for a vast class of constraints arising in robotic applications (possibly implicitly defined), and (3) provides effective and reliable practical performance. Specifically, SCP entails successively convexifying the cost function and constraints of a nonconvex optimal control problem, seeking a solution to the original problem through a sequence of convex problems. Its attractiveness is due to high computational speed \cite{AugugliaroSchoelligEtAl2012,SchulmanDuanEtAl2014}, broad applicability \cite{LiuLu2014,Virgili-llopZagarisEtAl2019}, and  (continuous-time) theoretical guarantees \cite{BonalliCauligiEtAl2019,MaoSzmukEtAl2016}. Extending SCP-based methods, primarily developed for Euclidean settings, to manifold-constrained problems is, however, challenging. In particular, when dealing with manifolds, it is challenging to make linearizations (required by SCP schemes that operate on dynamics) well-posed \cite{BulloLewis2007}. The key technical idea of this paper is to leverage {\em geometric embeddings} -- that is, mappings that allow one to recover manifolds as subsets of Euclidean spaces. Leveraging embeddings provides four main advantages. First, it allows one to lift a manifold-constrained problem into an equivalent problem defined over a space enjoying Euclidean structure, where linearizations can be easily computed. Second, embedded problems are often easier to address than their counterparts in local coordinates since, for example, linearity can be partially maintained. Third, embeddings provide a pathway to address implicitly-defined manifolds, as the equality constraints defining them are automatically satisfied in the lifted Euclidean space without the need for explicit enforcement. Fourth, and crucially, for dynamical systems evolving on Lie groups (as it is the case for virtually all robotic systems), there is always a ``natural" embedding that can, at least in principle, be leveraged. Indeed,  any mechanical system can always be identified with a subgroup $G \subseteq \mathbb{R}^{n \times n}$ of a Cartesian product of Lie groups of matrices. Then, the first-order equation that governs the dynamical evolution of the system is
\begin{equation} \label{ref:LieDynamics}
\dot{v}(t) = \Gamma(v(t)) + \sum^{m}_{i=1} u_i(t) \ \textnormal{vlft}(F_i)(v(t)), \quad v \in TG ,
\end{equation}
where $\Gamma$ is the geodesic spray induced by the kinetic energy and $\textnormal{vlft}(F_i)$ is the vertical lift of the generalized force $F_i$ \cite{BulloLewis2004}. This can be reinterpreted as a first-order, control-affine equation with a drift term on the space $G \times \mathfrak{g}$ via the identification $TG \cong G \times \mathfrak{g}$ \cite{BulloLewis2004}, where $\mathfrak{g}$ denotes the Lie algebra of $G$. A natural embedding is then the inclusion $ G \times \mathfrak{g}\subseteq \mathbb{R}^{2 n^2}$ (we provide explicit examples in the rest of the paper). 

{\em Statement of Contributions:} In this paper, we leverage geometric embeddings to extend SCP-based methods for trajectory optimization to manifold-constrained problems. Specifically, the contribution of this paper is fourfold. First, we introduce the notion of embedded SCP, a trajectory optimization method that exploits geometric embeddings to recast optimization on manifolds as a sequence of convex optimal control problems within Euclidean spaces. Importantly, a large number of trajectory optimization problems can be ``naturally" (in the sense above) embedded in Euclidean spaces, which makes this step generally straightforward. Second, leveraging such a reformulation and extending recent results on theoretical guarantees for SCP-based methods in Euclidean spaces \cite{BonalliCauligiEtAl2019}, we provide convergence guarantees for embedded SCP in the sense of the geometric Pontryagin Maximum Principle (PMP) \cite{Pontryagin1987}, i.e., in the sense of convergence of both the solution and corresponding Lagrange multipliers to stationary points satisfying necessary conditions for optimality and complying with the structure of the manifold characterizing the problem. In particular, a key aspect of our theoretical analysis entails showing how one can avoid manifold-type constraints in the optimization process, and yet can still guarantee that the computed solution lies on the manifold -- thus providing a computationally efficient pathway to deal with implicit manifold constraints. Third, by merging techniques from indirect optimal control and differential geometry, we extend the theoretical results to a large variety of settings, e.g., goal region constraints and  pointwise state constraints arising in multitask scenarios. Fourth, again inspired by analogue results in the Euclidean setting \cite{BonalliCauligiEtAl2019}, we harness the insights gained through our theoretical analysis to develop a convergence acceleration scheme for trajectory optimization on manifolds based on shooting methods \cite{Betts1998}.

One must note that while geometric embeddings allow a principled and systematic development of SCP methods for manifold-constrained problems, they come with two key drawbacks. First, embedding the manifold into an Euclidean space requires solving an optimization problem on a space having a higher dimension than the original manifold (albeit of simpler structure). Nevertheless, SCP-based algorithms scale rather well with problem dimensionality, and thus this drawback is offset by the simplification in the problem structure. Second, a straightforward and simple leveraging of embeddings is possible only if globally defined dynamical equations are at our disposal, i.e., it is easy to write
\begin{equation} \label{ref:EuclDyn}
\dot{x} = f(x,u), \ x \in \mathbb{R}^N \ \textnormal{with the additional constraint} \ x \in M
\end{equation}
for some $N \in \mathbb{N}$, where $M$ is a $n$-dimensional manifold. Indeed, though recovering the expression above is always theoretically possible thanks to \eqref{ref:LieDynamics}, it could be hard to practically describe complex dynamics as in \eqref{ref:EuclDyn} (for example, for second-order dynamics that are defined only by local coordinates). Still, most of the systems commonly used in robotics are in the form provided by \eqref{ref:EuclDyn}, which makes possible to efficiently put in practice the method developed in this paper.

{\em Structure of the Paper:} In Section \ref{sec:formulation}, we define the problem of trajectory optimization on manifolds and introduce the tool of embeddings that allows us to reformulate the problem in a standard Euclidean space on which we can proceed by linearization. In Section \ref{sec:algorithmguarantees}, we introduce the embedding trajectory optimization algorithm and the primary theoretical contributions that are geometrically consistent with the structure of the manifold. Finally, Sections \ref{sec:experiments} and \ref{sec:conclusions} present experiments, conclusions and future extensions for this work.

{\em Notation:} We denote by $M$ an $n$-dimensional manifold with tangent bundle $TM$ and cotangent bundle $T^*M$ (both manifolds of dimension $2n$). The tangent space of $M$ at $x \in M$ will be denoted by $T_xM$. Moreover, we recall that a smooth vector field on $M$ is a mapping $f : M \rightarrow TM$ such that $f(x) \in T_xM$, for every $x \in M$. The interested reader is referred to \cite{Lee2003} for related concepts in differential geometry.


\section{Problem Formulation and\\ Sequential Convex Programming on Manifolds}
\label{sec:formulation}

We begin in Section \ref{subsec:ocp} by formulating trajectory optimization for dynamical systems as an optimal control problem on manifolds. Then, in Section \ref{subsec:embedding}, we introduce a procedure for appropriately embedding the original problem on manifolds into an Euclidean space. This allows us to exploit classical SCP frameworks for trajectory optimization in Euclidean spaces to solve the problem in Section \ref{subsec:scp}.

\subsection{Trajectory Optimization on Generic Manifolds}
\label{subsec:ocp}

In this paper we consider a continuous-time formulation to ensure that the theoretical guarantees we derive are independent of the discretization scheme that is employed. A discussion about the impact of discretization schemes on the proposed methodology will be discussed in Section \ref{sec:geometryDiscret}.

Specifically, consider an initial point $\bar{x}_0 \in M$ and smooth mappings $g^i : M \rightarrow \mathbb{R}^{r_i}$, $i=1,\dots,\ell$, which are submersions at 0. Here, $g^i$ represent pointwise state constraints that are used to mathematically model multitask scenarios; in particular, $g^{\ell}$ represents goal region constraints. Without loss of generality, we require $\textnormal{dist}(\bar{x}_0,(g^{\ell})^{-1}(0)) > 0$, where $\textnormal{dist}$ is a point-set distance evaluated w.r.t. some Riemannian metric on $M$. For times $0 < t_1 < \dots < t_{\ell}$, we model the dynamical evolution of the system by the following drift control-affine system in $M$
\begin{eqnarray} \label{ref:AffineDynamics}
\begin{cases}
\displaystyle \dot{x}(t) = f(x(t),u(t)) = f_0(x(t)) + \sum_{j=1}^{m} u^j(t) f_j(x(t))\medskip \\
x(0) = \bar{x}_0 \quad , \quad g^i(x(t_i)) = 0 \ , \ i=1,\dots,\ell, \quad t \in [0,t_{\ell}],
\end{cases}
\end{eqnarray}
where $f_j : M \rightarrow TM$, $j = 0,\dots,m$ are $C^1$ vector fields. The pointwise state constraints $g^i(x(t_i)) = 0$  are useful in multitask scenarios where one seeks to jointly optimize subtrajectories connecting different waypoints. We emphasize that, as previously mentioned, the dynamics of every mechanical system can be written as in Eq. \eqref{ref:AffineDynamics} by substituting the manifold $M$ with its tangent bundle $TM$.

We pose trajectory optimization as an optimal control problem with penalized state constraints. Specifically, we define the Optimal Control Problem (\textbf{OCP}) as minimizing the integral cost
\begin{equation} \label{ref:Cost}
\begin{split}
J(t_{\ell},&x,u) = \int_{0}^{t_{\ell}} f^0(x(t),u(t)) \; \mathrm{d}t = \\
& \int_{0}^{t_{\ell}} \left( \| u(t) \|^2_R + u(t) \cdot f^0_u(x(t)) + g(x(t))\right) \; \mathrm{d}t
\end{split}
\end{equation}
under dynamics and pointwise constraints \eqref{ref:AffineDynamics}, among all control trajectories $u \in L^{\infty}([0,t_{\ell}], \, \mathbb{R}^m)$ satisfying $u(t) \in U$ almost everywhere in $[0,t_{\ell}]$, where the measurable set $U \subseteq \mathbb{R}^{m}$ represents control constraints. Here, $f^0_u : M \rightarrow \mathbb{R}^m$, $g : M \rightarrow \mathbb{R}$ are $C^1$, $\| \cdot \|_R$ is the weighted norm defined by a constant positive-definite matrix $R \in \mathbb{R}^{m \times m}$, and times $t_i$, $i=1,\dots,{\ell}$ are fixed. We remark that hard enforcement of dynamical and intermediate/final goal set constraints is naturally imposed by \eqref{ref:AffineDynamics}. The function $g = g_a + \omega g_b$ accumulates contributions from a purely state-dependant cost $g_a$ and state constraint penalty function $g_b$ (e.g., stemming from collision-avoidance constraints), weighted by $\omega \ge 1$. Penalizing state constraints (e.g., collision avoidance) provides both theoretical and numerical benefits: it allows us to obtain theoretical guarantees in the sense of the Pontryagin Maximum Principle \cite{Pontryagin1987,AgrachevSachkov2004,DmitrukKaganovich2011}, necessary conditions for optimality that are stronger than standard Lagrange multiplier rules (see also Theorem \ref{ref:theoSCP} below), and it provides numerical flexibility by often allowing simple trajectories that violate constraints such as obstacle avoidance to be exploited for initialization. Indeed, given correct design of an SCP algorithm, we can still guarantee that returned solutions satisfy state constraints up to a user-defined tolerance.

A representative example (that will serve as running example) of a trajectory optimization problem evolving on manifolds is the minimum-energy optimal control of a spacecraft avoiding collisions in a microgravity environment, which can be stated as (under the previous formalism, $g$ provides state constraints):
\begin{eqnarray*}
\textnormal{(\textbf{Ex})} \ \begin{cases}
\displaystyle \quad \min \ \int^{t_{\ell}}_0 \Big( \| u_1(t) \|^2 + \| u_2(t) \|^2 + g(r(t),q(t)) \Big) \; \mathrm{d}t \bigskip \\
\displaystyle \ \dot{r} = v \ , \ \dot{v} = \frac{u_1}{m} \ , \ \dot{q} = \frac{1}{2} \Omega(w) q \ , \ \dot{w} = J^{-1}(u_2 - w \times J w) \medskip \\
\ (r,v,q,w)(0) = \bar x_0 \ , \ (r,v,q,w)(t_{\ell}) = \bar x_{\ell} \ , \ (u_1,u_2) \in U \medskip \\
\ (r,v,q,w) \in \mathbb{R}^6 \times S^3 \times \mathbb{R}^3 \ , \ S^3 = \{ q \in \mathbb{R}^4 : \| q \| = 1 \}
\end{cases}
\end{eqnarray*}
where $r$ is the position of the vehicle, $v$ its tangential velocity, $q$ its orientation (expressed via quaternions), $w$ its angular velocity ($\Omega(w)$ is the usual skew-symmetric matrix depending on $w$) and the manifold $S^3 \subseteq \mathbb{R}^4$ characterizes quaternions. Controls are represented by the thrust $u_1$ and the torque $u_2$. A na\"ive way to approach our running example (\textbf{Ex}) would entail removing, without principled justification, the constraint $q \in S^3$, and then solving the relaxed problem in the resulting Euclidean space -- this could result, however, in computation of infeasible trajectories. Better justified approaches could exploit local charts or Lie group properties, as mentioned in Section \ref{sec:introduction}. However, in what follows, we demonstrate another method to tackle the implicit manifold constraint $q \in S^3$, that hinges on embeddings and provides a way to lift SCP methods to manifold-constrained problems.

\subsection{Embedding the Problem into the Euclidean Space} \label{subsec:embedding}

We would like to solve (\textbf{OCP}) via SCP, i.e., by an iterative procedure based on the linearization of all nonlinear mappings around the solution at the previous iteration. This requires us to compose a notion of linearized vector fields of $M$ around curves.
In \cite{Sussman1998,LewisTyner2003,BulloLewis2007}, the authors adapt such a definition on manifolds by recasting differential equations as algebraic equations of operators in $TM$. However, in many applications concerning dynamical systems, $M$ naturally appears as subset of the Euclidean space, in which case the most intuitive linearization is the one operating in the ambient Euclidean space. This insight motivates our approach: to recast (\textbf{OCP}) into an appropriate Euclidean space via geometric embeddings, i.e., mappings $e : M \rightarrow \mathbb{R}^N$ for $N \in \mathbb{N}$, and then to linearize in the ambient space.

Following the previous discussion, we assume that $M$ is a closed submanifold of $\mathbb{R}^N$, for some $N \in \mathbb{N}$. This means that we fix a particular embedding, which is given by the canonical inclusion $e : M \rightarrow \mathbb{R}^N : x \mapsto x$. This choice is made without loss of generality because, due to Whitney-type theorems \cite{Lee2003}, such a mapping always exists. Moreover, consistent with the previous motivating discussion, we assume that the mappings defining (\textbf{OCP}) naturally extend to $\mathbb{R}^N$, i.e., there exist $C^1$ vector fields $F_j : \mathbb{R}^N \rightarrow T\mathbb{R}^N \cong \mathbb{R}^N$, $j=1,\dots,m$ (with an abuse of notation, see \cite{Lee2003}), and $C^1$ functions $F^0_u : \mathbb{R}^N \rightarrow \mathbb{R}$, $G : \mathbb{R}^N \rightarrow \mathbb{R}$ and $G^i : \mathbb{R}^N \rightarrow \mathbb{R}$, $i=1,\dots,{\ell}$, that are smooth submersions at 0, such that $F_j|_M = f_j$, $G^i|_M = g^i$, $F^0_u|_M = f^0_u$ and $G|_M = g$. This setup allows us to transform the dynamics \eqref{ref:AffineDynamics} into the following drift control-affine system in $\mathbb{R}^N$:
\begin{eqnarray} \label{ref:AffineDynamicsRN}
\begin{cases}
\displaystyle \dot{x}(t) = F(x(t),u(t)) = F_0(x(t)) + \sum_{j=1}^{m} u^j(t) F_j(x(t))\medskip \\
x(0) = \bar{x}_0 \in M \subseteq \mathbb{R}^N \quad , \quad G^i(x(t_i)) = 0 \ , \ i=1,\dots,{\ell} .
\end{cases}
\end{eqnarray}
Therefore, (\textbf{OCP}) can be embedded in $\mathbb{R}^N$ by considering the following Embedded Optimal Control Problem (\textbf{EOCP}), which consists of minimizing the integral cost
\begin{equation} \label{ref:CostRN}
\begin{split}
J(t_{\ell},&x,u) = \int_{0}^{t_{\ell}} F^0(x(t),u(t)) \; \mathrm{d}t = \\
& \int_{0}^{t_{\ell}} \left( \| u(t) \|^2_R + u(t) \cdot F^0_u(x(t)) + G(x(t))\right) \; \mathrm{d}t
\end{split}
\end{equation}
under dynamics \eqref{ref:AffineDynamicsRN}, among all control trajectories $u \in L^{\infty}([0,t_{\ell}], \, \mathbb{R}^m)$ satisfying $u(t) \in U$ almost everywhere in $[0,t_{\ell}]$. At this step, it is worth noting that this embedding approach is justified only if solving (\textbf{EOCP}) is equivalent to solve (\textbf{OCP}). Fortunately, this is actually the case: every couple $(x,u)$ is optimal for (\textbf{OCP}) if and only if it is optimal for the embedded problem (\textbf{EOCP}). The validity of the whole scheme hinges on this crucial remark, which is summarized in the statement below,

\begin{lemma}[Embedding Lemma] \label{ref:Microlemma}
A tuple $(x,u)$ is optimal for (\textbf{EOCP}) if and only if it is optimal for (\textbf{OCP}).
\end{lemma}
\begin{proof}
The proof makes use of standard tools in differential geometry, whose definitions can be found in \cite{Lee2003} and are omitted here due to space limitations. We retrace the main steps of the proof. Denote by $e : M \rightarrow \mathbb{R}^N$ the canonical inclusion. It follows that $F$ is $e$-correlated to $f$, which implies their flows satisfy $\Phi_F(t,x) = \Phi_F(t,e(x)) = e(\Phi_f(t,x)) = \Phi_f(t,x)$ for every $(t,x) \in \mathbb{R} \times M$ for which the flow $\Phi_f$ is defined. Since $\bar{x}_0 \in M$, we obtain that $(x,u)$ satisfies the dynamics in \eqref{ref:AffineDynamicsRN} if and only if it satisfies the embedded dynamics in \eqref{ref:AffineDynamics}. From this, the optimality of $(x,u)$ follows from the similarity between \eqref{ref:Cost} and \eqref{ref:CostRN}.
\end{proof}

\begin{remark}\label{rem:hard_dyn}
Crucially, from Lemma \ref{ref:Microlemma}, we see that the satisfaction of implicit manifold-type constraints for (\textbf{EOCP}) is induced by hard enforcement of dynamical constraints. Therefore, any numerical strategy used to solve (\textbf{EOCP}) must provide hard enforcement of dynamics -- otherwise, the solution trajectory is not guaranteed to lie on the manifold!
\end{remark}

Let us show how this embedding framework applies to our running example  (\textbf{Ex}). It is sufficient to note that the only components of the dynamics evolving on a manifold are given by the mapping
\begin{equation*}
\displaystyle S^3 \times \mathbb{R}^3 \subseteq \mathbb{R}^7 \rightarrow T(S^3 \times \mathbb{R}^3) : (q,w) \mapsto \Big(\frac{1}{2}\Omega(w)q,J^{-1}(u_2 - w \times J w)\Big),
\end{equation*}
and that this mapping is also defined when $(q,w) \in \mathbb{R}^7$. In other words, the original dynamics is equivalent to $\Big(\frac{1}{2}\Omega(w)q,J^{-1}(u_2 - w \times J w)\Big)$ restricted to the subset $S^3 \times \mathbb{R}^3$. Therefore, the embedded dynamics related to this mapping are exactly the same but extended on $\mathbb{R}^7$, which shows that the embedded version of our example problem coincides with the original (\textbf{OCP}). Luckily, for many robotics applications which include trajectory optimization on manifolds, (\textbf{EOCP}) is equivalent to (\textbf{OCP}), which is also the case when formulation \eqref{ref:LieDynamics} is met. This is among the main motivations for developing such an embedded framework (see also our discussion at the end of Section \ref{subsec:ocp}).

\subsection{Reformulating Problem (\textbf{OCP}) via SCP in Euclidean Space} \label{subsec:scp}

Given that (\textbf{EOCP}) evolves in the Euclidean space, we may solve it using SCP.
Below, we describe a particular SCP formulation that enjoys geometrically consistent theoretical convergence guarantees.

Under the assumption that $U$ is convex, we iteratively linearize the nonlinear contributions of (\textbf{EOCP}) around local solutions, thus recursively defining a sequence of simplified problems. Specifically, at the end of iteration $k$, assume we have some continuous curves $x_k : [0,t_{\ell}] \rightarrow \mathbb{R}^N$ and $u_k : [0,t_{\ell}] \rightarrow \mathbb{R}^m$, continuously extended in the interval $(0,+\infty)$. Then, at iteration $k+1$, the Linearized Embedded Optimal Control Problem (\textbf{LEOCP})$_{k+1}$ consists of minimizing the new cost
\begingroup
\begin{multline} \label{ref:Cost_k}
J_{k+1}(t_{\ell},x,u) = \int_{0}^{t_{\ell}} F^0_{k+1}(t,x(t),u(t)) \; \mathrm{d}t = \medskip \\
\int_{0}^{t_{\ell}} \Big( \| u(t) \|^2_R + h_k(\| x(t) - x_k(t) \|^2 - \Delta_k) \Big) \; \mathrm{d}t \\
+ \int_{0}^{t_{\ell}} u(t) \cdot \left( F^0_{u}(x_k(t)) + \frac{\partial F^0_{u}}{\partial x}(x_k(t)) \cdot (x(t) - x_k(t)) \right) \; \mathrm{d}t \\
+ \int_{0}^{t_{\ell}} \left( G_k(x_k(t)) + \frac{\partial G_k}{\partial x}(x_k(t)) \cdot (x(t) - x_k(t)) \right) \; \mathrm{d}t
\end{multline}
\endgroup
where, consistent with the notation of Section \ref{subsec:embedding}, $G_k = G_a + \omega_k G_b$ and $h_k(s)$ is any smooth approximation of $\max\{0,s\}$ \cite[Chapter 10]{Lee2003}. Function $h_k$ provides trust-region guarantees on the {\em updates} in the state trajectories and constraints via the bounds $0 \le \Delta_k \le \Delta_0$ and weights $1 \le \omega_0 \le \omega_k \le \omega_{\max}$. The dynamical constraint for (\textbf{LEOCP})$_{k+1}$ is
\begingroup
\begin{eqnarray} \label{ref:AffineDynamics_k}
\begin{cases}
\displaystyle \dot{x}(t) = F_{k+1}(t,x(t),u(t)) = \medskip \\
\displaystyle\left( F_0(x_k(t)) + \sum_{j=1}^{m} u^j(t) F_j(x_k(t)) \right) \medskip \\
\displaystyle + \left( \frac{\partial F_0}{\partial x}(x_k(t)) + \sum_{j=1}^{m} u^j_k(t) \frac{\partial F_j}{\partial x}(x_k(t)) \right) \cdot (x(t) - x_k(t)) \medskip \\
x(0) = \bar{x}_0 \in M \quad , \quad G^i_{k+1}(x(t_i)) = G^i(x_k(t_i)) + \medskip \\
\displaystyle \hspace{20pt} \frac{\partial G^i}{\partial x}(x_k(t_i)) \cdot (x(t_i) - x_k(t_i)) = 0 \ , \ i=1,\dots,{\ell},
\end{cases}
\end{eqnarray}
\endgroup
obtained from the linearized expansion of all nonlinear mappings. We minimize among all controls $u \in L^{\infty}([0,t_{\ell}],  \, \mathbb{R}^m)$ satisfying $u(t) \in U$ almost everywhere in $[0,t_{\ell}]$. Inductively, the curves $x_{k+1}:[0,t_{\ell}]\rightarrow \mathbb{R}^N$ and $u_{k+1}: [0,t_{\ell}]\rightarrow \mathbb{R}^{m}$ are defined as the optimal solution for problem (\textbf{LEOCP})$_{k+1}$, continuously extended in the interval $(0,+\infty)$. Ideally, SCP algorithms may vary $\Delta_k$ and $\omega_k$ at each iteration to smoothen the process towards convergence, for example as in \cite{BonalliCauligiEtAl2019}.

A convexified formulation similar to \eqref{ref:Cost_k}-\eqref{ref:AffineDynamics_k} has already been introduced in \cite{BonalliCauligiEtAl2019}. However, we stress the fact that this new formulation deals with the presence of the manifold $M$ and of pointwise state constraints. The introduction of these two new features necessitates a considerable revision of the proof of theoretical guarantees (see the \iftoggle{ext}{Appendix}{extended version \cite{BonalliBylardEtAl2019}}).

The sequence of problems (\textbf{LEOCP})$_{k}$ is well-posed if, for each iteration $k \ge 1$, an optimal solution for (\textbf{LEOCP})$_{k}$ exists. For this, we consider the following assumptions:

\begin{itemize}
\item[$(A_1)$] The set $U$ is compact and convex. Moreover, the differentials of mappings $G^i$, $i=1,\dots,{\ell}$, are of full rank.
\item[$(A_2)$] Mappings $f^0$, $g$, vector fields $f_j$, $j = 0,\dots,m$ and their differentials have compact supports (and do $F^0$, $G$, $F_j$).
\item[$(A_3)$] At every iteration $k \ge 1$, problem (\textbf{LEOCP})$_{k}$ is feasible.
\end{itemize}

Under these assumptions, classical existence Filippov-type arguments \cite{Filippov1962,LeeMarkus1967} (applied to the reduced form of (\textbf{LEOCP})$_{k}$, see also the \iftoggle{ext}{Appendix}{extended version \cite{BonalliBylardEtAl2019}})
show that, at each iteration $k \ge 1$, the problem (\textbf{LEOCP})$_{k}$ has at least one optimal solution. We remark that similar assumptions have been considered in \cite{BonalliCauligiEtAl2019}; in the present contribution, $(A_1)$-$(A_3)$ gather the assumptions in \cite{BonalliCauligiEtAl2019} and appropriately adapt them to the context of manifolds and pointwise state constraints. Comments on their validity for very general trajectory optimization problems are easily adapted from \cite[Section II.B]{BonalliCauligiEtAl2019}.

Coming back to our running example problem (\textbf{Ex}), since we have already proved that the embedded problem coincides with (\textbf{Ex}), the linearization technique above applies directly to (\textbf{Ex}) without any additional step. This is particularly useful and happens every time the embedded problem is equivalent to the original one, which is common in trajectory optimization as highlighted previously. We remark that in nearly all scenarios having natural control constraints $U$, Assumptions $(A_1)$-$(A_3)$ are easily satisfied by (\textbf{Ex}).


\section{Algorithm Overview and Theoretical Guarantees} \label{sec:algorithmguarantees}
In Section \ref{subsec:algo}, we detail a general algorithm for the solution of (\textbf{OCP}) which combines SCP-based procedures with the embedded framework defined previously. Its convergence guarantees, in the sense of the Pontryagin Maximum Principle \cite{Pontryagin1987}, are studied in Section \ref{subsec:convGar} to \ref{sec:geometryDiscret}, where we show that these respect the original structure of the manifold in (\textbf{OCP}), despite solving a sequence of linearized versions of the embedded problem. Notably, this procedure allows one to solve (\textbf{OCP}) on manifolds defined implicitly via nonlinear equalities {\em without explicit representation}.

\subsection{SCP-based Trajectory Optimization on Manifolds} \label{subsec:algo}

\begin{algorithm}[h!]
\caption{Embedded SCP (E-SCP)}\label{ref:algoSCP}
\begin{algorithmic}[1]
\State {\bf Input: }{Trajectory $x_0$ and control $u_0$ defined in $(0,\infty)$.}
\State {\bf Output: }{Solution for (\textbf{LEOCP})$_{k}$ at iteration $k$.}
\State {\bf Data: }{Parameters for the used SCP procedure.}
\State    Transform (\textbf{OCP}) into problem (\textbf{EOCP}) in the Euclidean space, as in Section \ref{subsec:embedding};
 \State   Linearize (\textbf{EOCP}) by defining a sequence of convex problems (\textbf{LEOCP})$_{k}$ (an example of such linearization is given in Section \ref{subsec:scp});
 \State   Select a SCP procedure on Euclidean spaces to solve the sequence of problems (\textbf{LEOCP})$_{k}$ for $(x_{k},u_{k})$;
 \State   \Return{$(x_{k},u_{k})$ at the last iteration.}
\end{algorithmic}
\end{algorithm}

The first two steps in the algorithm above consist of transforming (\textbf{OCP}) into the optimal control problem (\textbf{EOCP}) on the Euclidean space via embedding procedures and successively linearizing it as detailed in Section \ref{subsec:embedding}. In the third step, one finally applies some SCP scheme on Euclidean spaces. It is important to remark that E-SCP provides the user with the freedom to choose any sequential convex procedure to solve the sequence of problems (\textbf{LEOCP})$_{k}$. However, we show in Section \ref{subsec:convGar} that specific choices of SCP (e.g. using hard enforcement of dynamical constraints) provide theoretical guarantees for E-SCP which are also consistent with the presence of the manifold within the original problem (\textbf{OCP}).

Problem (\textbf{LEOCP})$_{1}$ is linearized around an initial curve tuple $(x_0,u_0)$, where these initialization curves should be as close as possible to a feasible or even optimal curve for (\textbf{LEOCP})$_{1}$, although we do not require that $(x_0,u_0)$ is feasible for the embedded problem (\textbf{EOCP}). This allows one to initialize E-SCP with simple, even infeasible, guesses for solutions of (\textbf{EOCP}), such as a straight line in the manifold, as detailed in \cite[Section III.A]{BonalliCauligiEtAl2019}.

\subsection{Necessary Conditions for Optimality} \label{subsec:convGar}
Although the strategy for solving problems (\textbf{LEOCP})$_{k}$ in E-SCP is up to the user, we show that specific choices of solvers allow one to recover important theoretical guarantees for the convergence of E-SCP to critical points for the original problem (\textbf{OCP}) on manifolds. Specifically, we can show the convergence of E-SCP towards a trajectory satisfying first-order necessary conditions for optimality under the Pontryagin Maximum Principle \cite{Pontryagin1987} when the sequence of problems (\textbf{LEOCP})$_{k}$ is chosen as in Section \ref{subsec:scp}. However, we must adapt the proof for the classical Euclidean setting to take into account the presence of both manifold and pointwise state constraints. This will be done by leveraging fundamental results from differential geometry and optimal control, i.e., Hamiltonian systems and the Pontryagin Maximum Principle with pointwise state constraints. For self-containtment, we summarize some of these results in the following discussion (see, e.g., \cite{Lee2003,AgrachevSachkov2004,DmitrukKaganovich2011} for an extended treatment).

For consistency with the existing presentation of these results, denote $f^0$ as a time-varying function, i.e., $f^0 = f^0(t,x,u)$, and fix some measurable control time-series $u : \mathbb{R} \rightarrow \mathbb{R}^m$. Thanks to Assumptions $(A_1)$, $(A_2)$, for every $t_0 \in \mathbb{R}$, the trajectories arising for the augmented system
\begin{eqnarray} \label{ref:AugmentedAffineDynamics}
\begin{cases}
\displaystyle \dot{x}^0(t) = f^0(t,x(t),u(t)), \quad \dot{x}(t) = f(t,x(t),u(t))\medskip \\
(x^0,x)(t_0) = (0,y_0), \quad (x^0,x)(t) \in \mathbb{R} \times M
\end{cases}
\end{eqnarray}
exist in $\mathbb{R}$ for every $y_0 \in M$. Notice that $x^0(t)$ is simply the accumulated cost at time $t$. Therefore, the flow of system \eqref{ref:AugmentedAffineDynamics} is defined for every $(t_0,y_0) \in \mathbb{R} \times M$ and we denote it by $\exp_{(f^0_u,f_u)}(\cdot;\cdot,\cdot) : \mathbb{R}^2 \times M \rightarrow \mathbb{R} \times M$, such that $\exp_{(f^0_u,f_u)}(\cdot;t_0,y_0)$ represents the trajectory of \eqref{ref:AugmentedAffineDynamics} starting from $(0,y_0)$ at $t_0$.

By standard identifications, we denote $(p^0,p) \in T^*(\mathbb{R} \times M)$ and we define the Hamiltonian function related to \eqref{ref:AugmentedAffineDynamics} as
\begin{multline} \label{ref:AugmentedHamiltonian}
H : \mathbb{R} \times T^*(\mathbb{R} \times M) \times \mathbb{R}^m \rightarrow \mathbb{R} \\
(t,p^0,p,u) \mapsto \langle p , f(t,\pi(p),u) \rangle + p^0 f^0(t,\pi(p),u)
\end{multline}
where $\pi : T^*M \rightarrow M$ is the canonical projection and $\langle \cdot , \cdot \rangle$ denotes the duality in $T^*M$. As a classical result on Hamiltonian systems \cite{AgrachevSachkov2004,Lee2003}, for every $(t,u) \in \mathbb{R}^{m+1}$, one can uniquely associate to \eqref{ref:AugmentedHamiltonian} the so-called Hamiltonian vector field $\overset{\rightarrow}{H}(t,\cdot,\cdot,u) : T^*(\mathbb{R} \times M) \rightarrow T(T^*(\mathbb{R} \times M))$ by the rule $\sigma_{(p^0,p)}(\cdot,\overset{\rightarrow}{H}(t,\cdot,\cdot,u)) = d_{(p^0,p)}H(t,\cdot,\cdot,u)$, $\sigma$ being the canonical symplectic form of $T^*(\mathbb{R} \times M)$ (see, e.g., \cite{AgrachevSachkov2004}). Combining the classical geometric Pontryagin Maximum Principle \cite{AgrachevSachkov2004} with the reduction scheme for pointwise state constraints developed in \cite{DmitrukKaganovich2011} yields the following extended geometric Pontryagin Maximum Principle (see \iftoggle{ext}{Appendix}{extended version \cite{BonalliBylardEtAl2019}} for a proof sketch).

\begin{theorem}[Geometric Pontryagin Maximum Principle with Pointwise State Constraints] \label{ref:theo_PMPGEO}
Let $x$ be an optimal trajectory for (\textbf{OCP}), associated with the control $u$ in $[0,t_{\ell}]$ and with fixed interior and final times $t_i$, $i = 1,\dots,{\ell}$. There exists a nonpositive constant scalar $p^0$ and a piecewise absolutely continuous function $p : [0,t_{\ell}] \rightarrow T^*M$, called the adjoint vector, satisfying $p(t) \in T^*_{x(t)}M$, with $(p^0,p) \neq 0$, such that, almost everywhere in $[0,t_{\ell}]$, the following relations hold:
\begin{itemize}
\item \textbf{Adjoint Equations}
\begin{eqnarray} \label{ref:adjointEqGEO}
\displaystyle \frac{d(p^0,p)}{dt}(t) = \overset{\rightarrow}{H}(t,p^0,p(t),u(t))
\end{eqnarray}
\item \textbf{Maximality Condition}
\begingroup
\begin{equation} \label{ref:maxCondGEO}
H(t,p^0,p(t),u(t)) = \underset{u \in U}{\max} \ H(t,p^0,p(t),u)
\end{equation}
\endgroup
\item \textbf{Transversality Conditions}

For each $i=1,\dots,{\ell}-1$, the adjoint vector satisfies
\begin{equation} \label{ref:trasv1GEO}
\displaystyle p(t^-_i) - p(t^+_i) \perp \textnormal{ker} \ \frac{\partial g^i}{\partial x}(x(t_i)) \ , \ p(t_{\ell}) \perp \textnormal{ker} \ \frac{\partial g^{\ell}}{\partial x}(x(t_{\ell})) .
\end{equation}
\end{itemize}
The tuple $(p^0,p,u)$ is a (Pontryagin) extremal. We say that $(p^0,p,u)$ is normal if $p^0 \neq 0$, and abnormal otherwise.
\end{theorem}
It is important to recall that the statement of optimality in Theorem \ref{ref:theo_PMPGEO} supersedes the classical Euclidean case as it additionally addresses pointwise state constraints (which consequently imply discontinuity of the adjoint vector), in context of {\em nonlinear} manifolds.  
These features considerably complexify the proof (cf. \iftoggle{ext}{Appendix}{extended version \cite{BonalliBylardEtAl2019}}).

\subsection{Convergence with Geometric Consistency}
The convergence of E-SCP can be inferred by leveraging one further commonly adopted regularity assumption concerning optimal controls:

\begin{itemize}
\item[$(A_4)$] At every iteration $k \ge 1$, the optimal control $u_k$ of (\textbf{LEOCP})$_{k}$ is piecewise continuous in every subinterval $[0,t_1]$ and $[t_{i-1},t_{i}]$ for $i=2,\dots,{\ell}$.
\end{itemize}
The key convergence result is stated next:

\begin{theorem}[Convergence Theorem] \label{ref:theoSCP}
Suppose that $(A_1)$-$(A_4)$ hold. Given any sequence of trust region radii and weights $((\Delta_k,\omega_k))_{k \in \mathbb{N}} \subseteq [0,\Delta_0] \times [\omega_0,\omega_{\max}]$, let $((x_k,u_k))_{k \in \mathbb{N}}$ be any sequence such that, for every $k \ge 1$, $(x_k,u_k)$ is optimal for (\textbf{LEOCP})$_{k}$ in $[0,t_{\ell}]$, where problems (\textbf{LEOCP})$_{k}$ are built as detailed in Section \ref{subsec:scp}. Up to some subsequence:
\begin{itemize}
\item $x_k \rightarrow \tilde x \in C^0([0,t_{\ell}],M)$, for the strong topology of $C^0$, and
\item $u_k \rightarrow \tilde u \in L^{\infty}([0,t_{\ell}],U)$, for the weak topology of $L^2$,
\end{itemize}
where $(\tilde x,\tilde u)$ is feasible for the original problem (\textbf{OCP}). Moreover, there exists a nonpositive constant scalar $\tilde p^0$ and a piecewise absolutely continuous function $\tilde \gamma : [0,t_{\ell}] \rightarrow \mathbb{R}^N$, with $(\tilde p^0,\tilde \gamma) \neq 0$, such that the tuple $((\tilde p^0,\tilde p) = \textnormal{Pr}(\tilde p^0,\tilde x,\tilde \gamma),\tilde u)$, where $\textnormal{Pr}$ is the orthogonal projection of $T^*\mathbb{R}^{N+1}|_{\mathbb{R} \times M}$ onto $T^*(\mathbb{R} \times M)$ (cf. \iftoggle{ext}{Appendix}{extended version \cite{BonalliBylardEtAl2019}}), represents a geometric Pontryagin extremal (in the sense of Theorem \ref{ref:theo_PMPGEO}) for the original problem (\textbf{OCP}) on $M$. In particular, as $k$ tends to infinity, up to some subsequence:
\begin{itemize}
\item $p^0_k \rightarrow \tilde p^0$, and
\item $\gamma_k|_{[t_i,t_{i+1})} \rightarrow \tilde \gamma|_{[t_i,t_{i+1})} \in C^0([t_i,t_{i+1}),\mathbb{R}^N)$, $i=1,\dots,{\ell}-1$, and $\gamma_k|_{[t_{{\ell}-1},t_{\ell}]} \rightarrow \tilde \gamma|_{[t_{{\ell}-1},t_{\ell}]} \in C^0([t_{{\ell}-1},t_{\ell}],\mathbb{R}^N)$ for the strong topology of $C^0$,
\end{itemize}
where $(p^0_k,x_k,\gamma_k,u_k)$ is a Pontryagin extremal of (\textbf{LEOCP})$_{k}$.
\end{theorem}
Due to space limitations, we present the proof of Theorem \ref{ref:theoSCP} in the \iftoggle{ext}{Appendix}{extended version \cite{BonalliBylardEtAl2019}}. In short, Theorem \ref{ref:theoSCP} asserts that there exists a sequence of solutions for problems (\textbf{LEOCP})$_{k}$ that converges (under appropriate topologies) to a critical point for (\textbf{OCP}), in the (strong) sense of Theorem \ref{ref:theo_PMPGEO}. Importantly, the limiting trajectory lies on $M$, despite solving the linearized, embedded versions (\textbf{LEOCP})$_{k}$ (see Figure \ref{fig:surface}). The termination properties of Algorithm E-SCP are stated next:
\begin{figure}[t!]
\centering
\includegraphics[width=0.85\columnwidth]{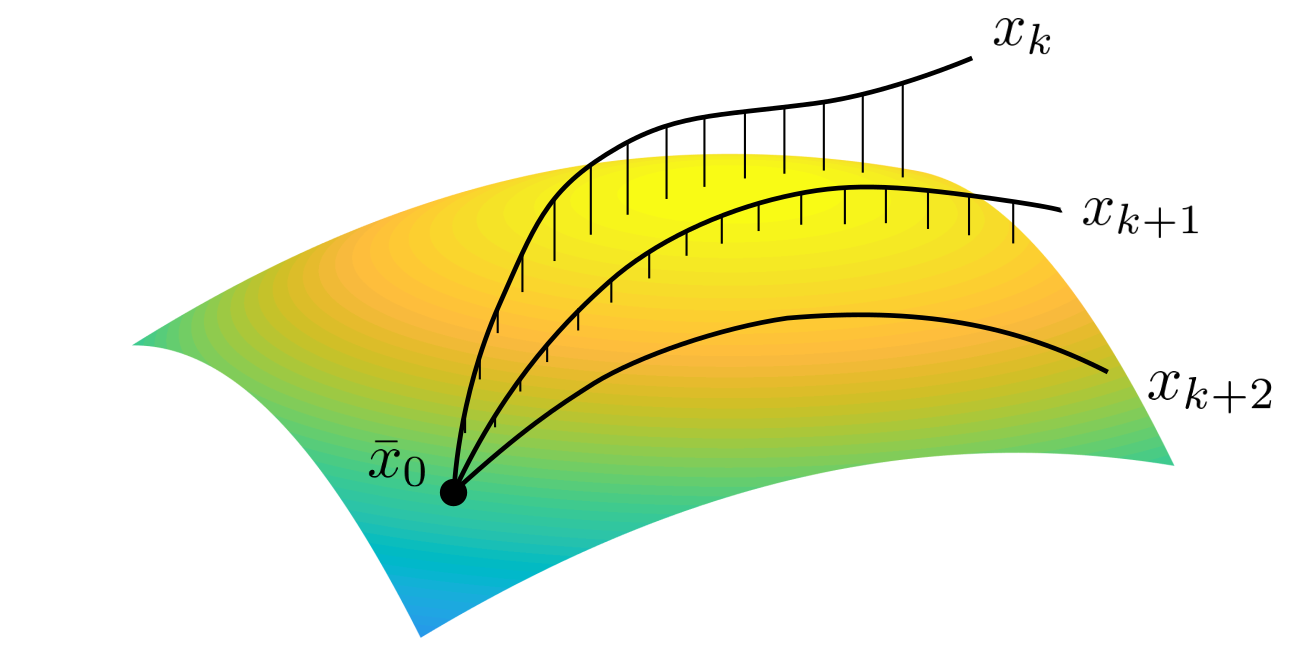}
\caption{(\textbf{EOCP}) uses hard enforcement of dynamical constraints to guarantee convergence to trajectories which lie on the manifold of the system, without explicitly enforcing or penalizing manifold constraints.} \label{figTorus}
\label{fig:surface}
\vspace{-15pt}
\end{figure}
\begin{corollary}[E-SCP Termination] \label{ref:corollSCP}
Assume (i) $(A_1)$-$(A_4)$ hold, (ii) problems (\textbf{LEOCP})$_{k}$ are built as detailed in Section \ref{subsec:scp}, (iii) the SCP procedure adopted for E-SCP (line \textbf{6} of Algorithm \ref{ref:algoSCP}) enforces hard dynamical constraints (as opposed to a penalized implementation -- see Remark~\ref{rem:hard_dyn}), (iv) $\Delta_k \in [0,\Delta_0]$ for all $k$ and $\omega_k < \omega_{\max}$ are chosen such that for every iteration $k \ge 1$, an output $(x_k,u_k)$ is always provided, and (v) the algorithm is terminated if $\omega_k > \omega_{\max}$. Then, in solving (\textbf{OCP}) by E-SCP only three mutually exclusive situations arise:
\begin{enumerate}
\item There exists an iteration $k \ge 1$ for which $\omega_k > \omega_{\max}$. Then, E-SCP terminates, providing a solution for (\textbf{LEOCP})$_{k}$ satisfying only soft state constraints.
\item There exists an iteration $k \ge 0$ for which $(x_{k+1},u_{k+1}) = (x_{k},u_{k})$. Then, E-SCP terminates, providing a stationary point, in the sense of the Pontryagin Maximum Principle, for the original problem (\textbf{OCP}).
\item We have $(x_{k+1},u_{k+1}) \neq (x_{k},u_{k})$, for every iteration $k \ge 0$. Then, E-SCP builds a sequence of optimal solutions for (\textbf{LEOCP})$_{k}$ that has a subsequence converging (with respect to appropriate topologies) to a stationary point, in the sense of the Pontryagin Maximum Principle, for the original problem (\textbf{OCP}).
\end{enumerate}
\end{corollary}
\begin{proof}
\noindent The assumptions on the internal SCP procedure for E-SCP imply that only these three cases may happen and that they are mutually exclusive. Moreover, case 3) is a direct consequence of Theorem \ref{ref:theoSCP} and, for case 2), it is sufficient to apply Theorem \ref{ref:theoSCP} to the sequence of solutions
$$
\Big( (x_0,u_0),(x_1,u_1),\dots,(x_{k-1},u_{k-1}),(x_k,u_k),(x_k,u_k),(x_k,u_k),\dots \Big)
$$
because it clearly converges to $(\tilde x,\tilde u) = (x_k,u_k)$.
\end{proof}
In case 1), SCP fails because no feasible strategies can be computed, which also occurs in other state-of-the-art trajectory optimization solvers such as TrajOpt \cite{SchulmanDuanEtAl2014}. Since the convergence of numerical methods usually leverages a termination threshold, case 3) is the most common outcome. Fortunately, Theorem \ref{ref:theoSCP} ensures that we are converging to a stationary point that satisfies strong necessary conditions for optimality in the sense of Theorem \ref{ref:theo_PMPGEO} and manifold-type constraints.

\subsection{Geometric Consistency and Discrete-Time Convergence} \label{sec:geometryDiscret}

Despite the fact that E-SCP entails solving a sequence of linearized problems for the embedded reformulation (\textbf{EOCP}), i.e., without explicit representation of the manifold, Lemma \ref{ref:Microlemma} and Corollary \ref{ref:corollSCP} ensure that the numerical solution converges to trajectories that satisfy the manifold constraints and also ensure that the limiting solution satisfies strong first-order necessary conditions for optimality that respect the geometric structure of the original manifold. 

The additional advantage of working within a continuous-time setting is that the validity of the theoretical guarantees is independent of the time-discretization scheme used to solve the ODEs. For instance, choosing schemes such as Euler or Simpson's rule lead to well well-posed convex optimization problems when the dynamics are linearized within each SCP step. With a sufficiently small time-step, one can ensure that the discrete solution provided by SCP stays close to the solution of the continuous-time problem (\textbf{EOCP}). Finally, Corollary \ref{ref:corollSCP} further implies that this discretized solution stays close to the solution for the original continuous-time (\textbf{OCP}).

\section{Convergence Acceleration via\\ Differential Shooting Method} \label{subsec:shooting}

An important result provided by Theorem \ref{ref:theoSCP} is the convergence of Pontryagin extremals related to the sequence of solutions for problem (\textbf{LEOCP})$_{k}$ in the Euclidean space towards a Pontryagin extremal related to the solution for (\textbf{OCP}) in the manifold, found by E-SCP. As a consequence, we can extend the acceleration procedure proposed in \cite{BonalliCauligiEtAl2019} to the manifold and pointwise state constraint case, i.e., warm-starting shooting methods \cite{Betts1998} with E-SCP.

The key idea is that since the convergence of the adjoint vectors is provided in $\mathbb{R}^N$, one can leverage them to re-state a shooting method on the manifold within $\mathbb{R}^N$. However, unlike the framework proposed in \cite{BonalliCauligiEtAl2019}, the main difficulty concerns the pointwise state constraints that introduce discontinuities for the multipliers. Fortunately, we can still use this hybrid method by leveraging the knowledge of adjoint vectors at intermediate times. Assuming SCP is converging, the Lagrange multipliers $\lambda^i_k$ related to the pointwise condition $G^{i}_k(x(t_i)) = 0$ for the finite dimensional discretization of problems (\textbf{LEOCP})$_{k}$ approximate the values $\gamma_k(t_i)$ of the adjoint vectors related to the continuous-time (\textbf{LEOCP})$_{k}$ (see \cite{GollmannKernEtAl2008} and the problem reduction provided in the \iftoggle{ext}{Appendix}{extended version \cite{BonalliBylardEtAl2019}}). Then, up to some subsequence, for every small $\delta > 0$, there exists an iteration $k_{\delta} \ge 1$ for which, for every iteration $k \ge k_{\delta}$, one has $\| \tilde \gamma(t_i) - \lambda^i_k \| < \delta$, $i=1,\dots,{\ell}$, where $\tilde \gamma$ is an adjoint vector related to the solution of (\textbf{OCP}) found by SCP (see Theorem \ref{ref:theoSCP}). This means that, starting from some iteration $k \ge k_{\delta}$, we can run a shooting method to solve (\textbf{OCP}), initializing using $\lambda^i_k$, $i=1,\dots,{\ell}$. At each iteration of SCP, we use the values $\lambda^i_k$ provided by the solver to initialize the shooting method until convergence is achieved. This provides a theoretically guaranteed method to accelerate convergence for SCP towards a more accurate solution.


\section{Numerical Experiments and Discussion}
\label{sec:experiments}
In this section, we provide implementation details and examples to demonstrate various facets of our approach. We focus on three important aspects: (1) providing comparisons between E-SCP and standard SCP on Euclidean spaces, (2) providing comparisons between E-SCP and state-of-the-art algorithms for trajectory optimization, and (3) analyzing convergence acceleration for E-SCP via shooting methods.

Simulations are provided by considering two problems: our running example (\textbf{Ex}), defined in Section \ref{subsec:ocp}, and a trajectory optimization problem for a 7 degree-of-freedom manipulator in a cluttered environment. Consistent with our embedding framework, rather than describing the manipulator via joint angle variables (i.e., local variables), the states are characterized by tuples $(x_i,y_i) \in \mathbb{S}^1$, $i=1,\dots,7$ (one for each joint), so that the system evolves in the $7$-dimensional torus, i.e., $(x_1,y_1,\dots,x_7,y_7) \in \mathbb{T}^7$, which is naturally embedded in $\mathbb{R}^{14}$. The dynamics are given by the kinematic equations of a manipulator, that is, each joint $i$ satisfies:
\begin{equation} \label{ref:dynManip}
f_i : \mathbb{T}^1 \times \mathbb{R} \rightarrow T \mathbb{T}^1 : (x_i,y_i,u_i) \mapsto (-y_i u_i,x_i u_i)
\end{equation}
where each $u_i \in \mathbb{R}$ is a control variable. As for our running example (\textbf{Ex}), \eqref{ref:dynManip} naturally represents a dynamical system in $\mathbb{R}^2$, so that, problem (\textbf{EOCP}) coincides with (\textbf{OCP}).

The examples and algorithms presented in this work were implemented in the Julia programming language \cite{BezansonKarpinskiEtAl2012} using the GuSTO.jl package located at \texttt{\url{https://github.com/StanfordASL/GuSTO.jl}}, with optimization problems solved using Gurobi \cite{ios_gurobi2016}. We chose GuSTO \cite{BonalliCauligiEtAl2019} as the SCP procedure for E-SCP (line \textbf{4} of Algorithm \ref{ref:algoSCP}), so that Corollary \ref{ref:corollSCP} held.
For each compared SCP method, the continuous-time optimal control problem was discretized using a trapezoidal approximation of the dynamics, assuming a zero-order hold for the control, and the discrete-time cost considered for each problem was the energy $\sum_{k=1}^{d-1} ||u_k||_2^2 \Delta t$, where $d$ is the number of discretization points for the trajectory. Additionally, obstacle avoidance constraints served as our non-convex state constraints. For each set of simulations presented, we report results for 50 experiments with different start and goal configurations. A SCP trial is marked as successful if the algorithm converged and the resulting solution was collision-free. We used the Bullet Physics engine to calculate signed distances for obstacle avoidance constraints or penalties \cite{Coumans, Ericson2004}. For each experiment, the variables were initialized using straight-line initializations on the manifold.

\subsection{Comparison with State-of-the-Art}

In this section, we demonstrate the benefits obtained when using E-SCP to solve trajectory optimization problems on manifolds, rather than standard SCP approaches. More specifically, the main advantage of E-SCP is that, when considering hard enforcement of dynamical constraints, the limiting numerical solution is guaranteed to lie on the manifold, even if such a constraint is not explicitly enforced. This approach is in contrast to enforcing the nonlinear manifold-type equality constraints that appear when using standard Euclidean-based SCP approaches. Since added equality constraints increase the complexity of the problem and may adversely affect efficiency, removing manifold-type constraints in SCP provides greater flexibility in solving the sequential problems, given that these constraints are implicitly satisfied.

For this comparison, we considered the 7-DoF manipulator using 120 discretization points over a trajectory time of 30 seconds \cite{LaValleKuffner2000}. The achieved results are shown in Table \ref{table:fullresults}. Here, comparisons are given between E-SCP with GuSTO as the internal solver, TrajOpt \cite{SchulmanDuanEtAl2014} without any enforced manifold-type constraints, and versions of GuSTO and TrajOpt where the manifold-type constraints are penalized (denoted by SCP and TrajOpt-P, respectively, in Table \ref{table:fullresults}).

E-SCP has the best performance in reduction of dynamical constraint error (even if negligible). As can be expected, the methods that explicitly penalize the presence of the manifold achieve the highest precision for manifold-constraint satisfaction (even if also negligible). However, a keen analysis of Table \ref{table:fullresults} shows that this comes with the tradeoff that the additional state-constraint penalties produce a greater tendency to fall into high-cost local minima, since the penalization affects the way the internal convex optimization algorithm reduces the cost. Consequently, the resulting true cost of the non-penalizing algorithms is on average lower than their penalizing counterparts. In other words, E-SCP provides better optimal solutions than state-of-the-art penalization approaches.

In addition, we note that although unit-norm constraints like $x \in \mathbb{T}^n$ can be easily formulated as penalty expressions and satisfied through penalization, more complex manifold constraints (e.g., those associated with SO(3), closed-kinematic chains, etc.) are more difficult to formulate and satisfy in this way, and would be better handled using E-SCP. Indeed, from Theorem \ref{ref:theoSCP} we see that manifold constraint error in E-SCP scales with the dynamical constraint error, which in turn depends only on the discretization scheme. Thus, we expect the negligible manifold constraint error for E-SCP in the previous example to carry over to more complex manifolds.

\newcolumntype{F}{>{\centering\arraybackslash}p{1.9cm}}
\newcommand{\Centerstackmoda}[1]{\addstackgap[2pt]{\Centerstack{#1}}}
\newcommand{\Centerstackmodb}[1]{\addstackgap[4pt]{\Centerstack{#1}}}
\newcommand{\Centerstackmodc}[1]{\addstackgap[10pt]{\Centerstack{#1}}}
\newcolumntype{G}{>{\centering\arraybackslash}p{1.2cm}}
\newcolumntype{H}{>{\centering\arraybackslash}p{0.4cm}}
\newcolumntype{I}{>{\centering\arraybackslash}p{0.92cm}}

\begin{centering}
\begin{table}[t]
\small
\centering
\vspace{-5pt}
\begin{tabular}{| F | G | G | G | G |}
\hline
 & \Centerstackmodb{E-SCP} & \Centerstackmodb{SCP} & \Centerstackmodb{TrajOpt} & \Centerstackmodb{TrajOpt-P} \\
\hline
\Centerstackmodb{Dynamical {constraint error}} & \Centerstackmodb{{$1.9$}{$\times 10^{-3}$}} & \Centerstackmodb{{$5.8$}{$\times 10^{-3}$}} & \Centerstackmodb{{$8.5$}{$\times 10^{-3}$}} & \Centerstackmodb{{$4.8$}{$\times 10^{-3}$}} \\
\hline
\Centerstackmoda{{$\mathbb{T}^7$ manifold} {constraint error}} & \Centerstackmodb{{$3.9$}{$\times 10^{-2}$}} & \Centerstackmodb{{$3.6$}{$\times 10^{-6}$}} & \Centerstackmodb{{$1.6$}{$\times 10^{-2}$}} & \Centerstackmodb{{$3.1$}{$\times 10^{-6}$}} \\
\hline
\Centerstackmoda{True cost} & 0.105 & 1.0 & 0.215 & 0.792 \\
\hline
\Centerstackmoda{Computation time} & 0.466 & 0.129 & 0.321 & 1.0 \\
\hline
\end{tabular}
\caption{Averaged results of manipulator arm experiments, with normalized cost, computation times and non-normalized values for constraint errors.}
\vspace{-15pt}
\label{table:fullresults}
\end{table}
\end{centering}

\subsection{Convergence Acceleration via Shooting Methods}
In this section, we provide numerical simulations that highlight convergence benefits that are obtained when E-SCP is combined with shooting methods. In particular, we consider trajectory optimization of a spacecraft having the dynamics and embedding given in (\textbf{Ex}) and navigating through a highly cluttered environment, using 100 discretization points over 50-second trajectories. As shown in Table \ref{table:shooting}, the results are very promising: on average, the shooting method cuts significantly the number of SCP iterations required to converge to a trajectory, resulting in an overall 59.4\% increase in speed. The difference in performance is made more stark by the fact that the shooting method can occasionally converge in cases where SCP is unable to converge at all due to the use of naive in-collision straight-line initialization. Indeed, it is also interesting to note that in cases where the shooting method provides the final convergence, the final trajectory cost is always lower (if sometimes only slightly) than the cost returned by SCP alone, as the shooting method relies on the Newton method and thus tends to achieve a much higher proximity to the optimal on convergence than SCP.

\newcommand{\addstackgapmod}[1]{\addstackgap[2pt]{#1}}
\newcolumntype{D}{>{\centering\arraybackslash}p{1.8cm}}
\newcommand{\Centerstackmodd}[1]{\addstackgap[4pt]{\Centerstack{#1}}}
\newcolumntype{E}{>{\centering\arraybackslash}p{2.5cm}}

\begin{centering}
\begin{table}[h!]
\small
\centering
\vspace{-5pt}
\begin{tabular}{| E || D | D |}
\hline
 & \Centerstackmodd{E-SCP Only} & \Centerstackmodd{{E-SCP +} Shooting} \\
\hline
\addstackgapmod{SCP Iterations} & 32 & 13 \\
\hline
\addstackgapmod{Reported Cost} & 0.221 & 0.176  \\
\hline
\end{tabular}
\caption{Averaged results of experiments using a shooting method to accelerate the convergence of E-SCP, while resulting in lower trajectory cost than using E-SCP alone.}
\vspace{-15pt}
\label{table:shooting}
\end{table}
\end{centering}


\section{Conclusions}
\label{sec:conclusions}

In this paper we provided an SCP-based method for trajectory optimization with manifold constraints. Our key insight was to leverage geometric embeddings to lift a manifold-constrained trajectory optimization problem into an equivalent problem defined over a space enjoying Euclidean structure, where SCP can be readily applied. We derived sound theoretical guarantees and validated the proposed methodology via numerical experiments. Among other benefits, our method can easily accommodate implicitly-defined manifold constraints.

This work opens the field to many future avenues of research. First, we plan to study the setting with free final times $t_i$, $i=1,\dots,\ell$, both from a theoretical and numerical standpoint. Second, related to the previous direction, it is of interest to design accurate numerical schemes that can handle multi-shooting methods and free final time settings. Third, we would like to leverage techniques from Lie group theory to improve performance when applying E-SCP to such specific class of submanifolds. Finally, we plan to evaluate our method on hardware platforms, such as robotic manipulators and test beds for free-flying robotic spacecraft.


\bibliographystyle{plainnat}
\bibliography{main,ASL_papers}

\iftoggle{ext}{\section*{Appendix}
\label{sec:appendix}

\subsection{Proof of Theorem \ref{ref:theoSCP}}

The proof consists of two main steps. First, we apply the reduction scheme in \cite{DmitrukKaganovich2011} to each linearized problems (\textbf{LEOCP})$_k$: every convexified optimal control problem with pointwise state constraints is converted into an optimal control problems without such constraints but with higher dimension. We apply the SCP convergence result provided by \cite[Theorem III.1]{BonalliCauligiEtAl2019} to this class of problems, recovering an extremal for the embedded problem (\textbf{EOCP}). Then, we project the obtained adjoint vector getting a geometric extremal for the original optimal control problem (\textbf{OCP}).

\begin{remark} \label{remarkAdjoint}
Considering Darboux coordinates $(y,\xi)$ for $T^*(\mathbb{R} \times M)$, equations \eqref{ref:adjointEqGEO} locally take the common form \cite{AgrachevSachkov2004}
\begin{eqnarray} \label{ref:adjointDarboux}
\begin{cases}
\displaystyle \dot{y}(t) = \frac{\partial H}{\partial \xi}(t,p^0,y(t),\xi(t),u(t)) \medskip \\
\displaystyle \dot{\xi}(t) = -\frac{\partial H}{\partial y}(t,p^0,y(t),\xi(t),u(t)) \ .
\end{cases}
\end{eqnarray}
In particular, when $M = \mathbb{R}^N$, the adjoint equations \eqref{ref:adjointDarboux} are globally equivalent to the adjoint equations \eqref{ref:adjointEqGEO}. Moreover, the constancy of $p^0$ arises from \eqref{ref:adjointDarboux} and the fact that the Hamiltonian \eqref{ref:AugmentedHamiltonian} does not depend on the variable $x^0$. This facts will be used in what follows.
\end{remark}


\subsection{Reduction of (\textbf{LEOCP})$_k$ to Classical Optimal Control Problems and Analysis for the Convergence of Extremals}

We first reduce the linearized problems (\textbf{LEOCP})$_k$ to optimal control problems without pointwise state constraints. For sake of concision and without loss of generality, from now on we assume that $\ell = 2$, i.e. we have two times $t_1$, $t_2$. Considering the general case goes similarly (see also \cite{dmitruk2011maximum}). \\

Let $((x_{k+1},u_{k+1}))_{k \in \mathbb{N}}$ be the sequence of optimal solutions for (\textbf{LEOCP})$_{k+1}$ in $[0,t_2]$ given in Theorem \ref{ref:theoSCP}. We chop each $(x_{k+1},u_{k+1})$ by defining, for every $s \in [0,1]$ and every $k \in \mathbb{N}$,
\begin{equation} \label{ref:chopTrajectories}
\begin{split}
&a_{k+1}(s) = x_{k+1}(t_1 s) \ , \ b_{k+1}(s) = x_{k+1}((t_2 - t_1) s + t_1) \\
&v_{k+1}(s) = u_{k+1}(t_1 s) \ , \ w_{k+1}(s) = u_{k+1}((t_2 - t_1) s + t_1) .
\end{split}
\end{equation}
For every iteration $k$, consider the Reduced Linearized Embedded Optimal Control Problem (\textbf{RLEOCP})$_{k+1}$ in $\mathbb{R}^N$
\begin{eqnarray*}
\begin{cases}
\displaystyle \min \ \int^1_0 \Big( t_1 F^0_{k+1}(t_1 s,a(s),v(s)) \medskip \\
\hspace{45pt} (t_2 - t_1) F^0_{k+1}((t_2 - t_1) s + t_1,b(s),w(s)) \Big) \; \mathrm{d}s \bigskip \\
a'(s) = t_1 F_{k+1}(t_1 s,a(s),v(s))\medskip \ , \ s \in [0,1] \\
b'(s) = (t_2 - t_1) F_{k+1}((t_2 - t_1) s + t_1,b(s),w(s)) \medskip \\
a(0) = \bar{x}_0 \ , \ G^1_{k+1}(a(1)) = 0 \ , \ G^2_{k+1}(b(1)) = 0 \medskip \\
a(1) - b(0) = 0 \quad , \quad v(s) \ , \ w(s) \in U \ , \ s \in [0,1]
\end{cases}
\end{eqnarray*}
where the condition $a(1) - b(0) = 0$ translates into the continuity of admissible trajectories for (\textbf{LEOCP})$_{k+1}$ at $t_1$. Therefore, by relations \eqref{ref:chopTrajectories} and the definition of problems (\textbf{LEOCP})$_{k+1}$, each tuple $(a_{k+1},b_{k+1},v_{k+1},w_{k+1})$ is an optimal solution for problem (\textbf{RLEOCP})$_{k+1}$. Applying the classical Pontryagin Maximum Principle with transversality conditions \cite{Pontryagin1987} to each problem (\textbf{RLEOCP})$_{k+1}$ provides the existence of nontrivial tuples $(p^0_{k+1},\zeta_{k+1},\eta_{k+1})$, where $p^0_{k+1}$ are nonpositive constant and $\zeta_{k+1}$, $\eta_{k+1}$ are absolutely continuous functions in $[0,1]$, satisfying, a.e. in $[0,1]$,
\begingroup
\begin{multline} \label{ref:adjointRLEOCPk}
\zeta'_{k+1}(s) = -t_1 \bigg( \zeta^{\top}_{k+1}(s) \frac{\partial F_{k+1}}{\partial x}(t_1 s,a_{k+1}(s),v_{k+1}(s)) \medskip \\
+ p^0_{k+1} \frac{\partial F^0_{k+1}}{\partial x}(t_1 s,a_{k+1}(s),v_{k+1}(s)) \bigg)
\end{multline}
\begin{multline}
\eta'_{k+1}(s) = -(t_2 - t_1) \bigg( \medskip \\
\eta^{\top}_{k+1}(s) \frac{\partial F_{k+1}}{\partial x}((t_2 - t_1) s + t_1,b_{k+1}(s),w_{k+1}(s)) \medskip \\
+ p^0_{k+1} \frac{\partial F^0_{k+1}}{\partial x}((t_2 - t_1) s + t_1,b_{k+1}(s),w_{k+1}(s)) \bigg)
\end{multline}
\begin{multline}
t_1 \Big( \zeta_{k+1}(s) \cdot F_{k+1}(t_1 s,a_{k+1}(s),v_{k+1}(s)) \medskip \\
+ p^0_{k+1} F^0_{k+1}(t_1 s,a_{k+1}(s),v_{k+1}(s)) \Big) \medskip \\
+ (t_2 - t_1) \Big( \eta_{k+1}(s) \cdot F_{k+1}((t_2 - t_1) s + t_1,b_{k+1}(s),w_{k+1}(s)) \medskip \\
+ p^0_{k+1} F^0_{k+1}((t_2 - t_1) s + t_1,b_{k+1}(s),w_{k+1}(s)) \Big) \medskip \\
= \underset{(v,w) \in U^2}{\max} \ \bigg( t_1 \Big( \zeta_{k+1}(s) \cdot F_{k+1}(t_1 s,a_{k+1}(s),v) \medskip \\
+ p^0_{k+1} F^0_{k+1}(t_1 s,a_{k+1}(s),v) \Big) \medskip \\
+ (t_2 - t_1) \Big( \eta_{k+1}(s) \cdot F_{k+1}((t_2 - t_1) s + t_1,b_{k+1}(s),w) \medskip \\
+ p^0_{k+1} F^0_{k+1}((t_2 - t_1) s + t_1,b_{k+1}(s),w) \Big) \bigg)
\end{multline}
\begin{equation} \label{ref:transversalityRLEOCPk}
\begin{split}
\displaystyle \zeta_{k+1}(1) - &\eta_{k+1}(0) \ \perp \ \textnormal{ker} \frac{\partial G^1_{k+1}}{\partial x}(a_{k+1}(1)) \medskip \\
&\displaystyle \eta_{k+1}(1) \ \perp \ \textnormal{ker} \frac{\partial G^2_{k+1}}{\partial x}(b_{k+1}(1))
\end{split}
\end{equation}
\endgroup
which represent adjoint equations, maximality condition and transversality conditions on the adjoint vectors, respectively. \\

We now analyze convergences properties for extremals $(p^0_{k+1},a_{k+1},b_{k+1},\zeta_{k+1},\eta_{k+1},v_{k+1},w_{k+1})$ related to problems (\textbf{RLEOCP})$_{k+1}$. For this, we make use \cite[Theorem III.1]{BonalliCauligiEtAl2019}. \\

A slight modification to the proof of \cite[Theorem III.1]{BonalliCauligiEtAl2019}\footnote{\begingroup \footnotesize As a matter of fact, one needs to adapt the proof to take into account initial conditions of type $x(0) \in M_0$, where $M_0$ is a submanifold of $M$, the fact that initial and final goal sets are linearized and the fact that $F^0_{k+1}$, $F_{k+1}$ are now evaluated at $s_1 t$, $(t_2 - t_1)s + t_1$. Modifications arise only for \cite[Lemma V.2]{BonalliCauligiEtAl2019}, and straightforward computations show that the convergence of variation vectors still hold. Remark that, since the reduced problems (\textbf{RLEOCP})$_{k+1}$, (\textbf{RLEOCP}) have fixed final time, we just need that $M_0$, $M_f$ be closed, which is the case for problems (\textbf{RLEOCP})$_{k+1}$. \endgroup} ensures the existence of an extremal $(\tilde p^0,\tilde a,\tilde b,\tilde \zeta,\tilde \eta,\tilde v,\tilde w)$ for the Reduced Embedded Optimal Control Problem (\textbf{REOCP})
\begin{eqnarray*}
\begin{cases}
\displaystyle \min \ \int^1_0 \Big( t_1 F^0(a(s),v(s)) + (t_2 - t_1) F^0(b(s),w(s)) \Big) \; \mathrm{d}s \bigskip \\
a'(s) = t_1 F(a(s),v(s))\medskip \ , \ s \in [0,1] \\
b'(s) = (t_2 - t_1) F(b(s),w(s)) \medskip \\
a(0) = \bar{x}_0 \ , \ G^1(a(1)) = 0 \ , \ G^2(b(1)) = 0 \medskip \\
a(1) - b(0) = 0 \quad , \quad v(s) \ , \ w(s) \in U \ , \ s \in [0,1]
\end{cases}
\end{eqnarray*}
therefore satisfying the following, adjoint equations, maximality condition and transversality conditions, respectively,
\begingroup
\begin{multline} \label{ref:adjointRLEOCP}
\tilde \zeta'(s) = -t_1 \bigg( \tilde \zeta^{\top}(s) \frac{\partial F}{\partial x}(\tilde a(s),\tilde v(s)) + \tilde p^0 \frac{\partial F^0}{\partial x}(\tilde a(s),\tilde v(s)) \bigg)
\end{multline}
\begin{multline} \label{ref:adjointRLEOCP1}
\tilde \eta'(s) = -(t_2 - t_1) \bigg( \tilde \eta^{\top}(s) \frac{\partial F}{\partial x}(\tilde b(s),\tilde w(s)) \medskip \\
+ \tilde p^0 \frac{\partial F^0}{\partial x}(\tilde b(s),\tilde w(s)) \bigg)
\end{multline}
\begin{multline} \label{ref:maxRLEOCP}
t_1 \Big( \tilde \zeta(s) \cdot F(\tilde a(s),\tilde v(s)) + \tilde p^0 F^0(\tilde a(s),\tilde v(s)) \Big) \medskip \\
+ (t_2 - t_1) \Big( \tilde \eta(s) \cdot F(\tilde b(s),\tilde w(s)) + \tilde p^0 F^0(\tilde b(s),\tilde w(s)) \Big) \medskip \\
= \underset{(v,w) \in U^2}{\max} \ \bigg( t_1 \Big( \tilde \zeta(s) \cdot F(\tilde a(s),v) + \tilde p^0 F^0(\tilde a(s),v) \Big) \medskip \\
+ (t_2 - t_1) \Big( \tilde \eta(s) \cdot F(\tilde b(s),w) + \tilde p^0 F^0(\tilde b(s),w) \Big) \bigg)
\end{multline}
\begin{multline} \label{ref:transversalityRLEOCP}
\displaystyle \tilde \zeta(1) - \tilde \eta(0) \perp \textnormal{ker} \frac{\partial G^1}{\partial x}(\tilde a(1)) \ , \ \tilde \eta(1) \perp \textnormal{ker}\frac{\partial G^2}{\partial x}(\tilde b(1)) .
\end{multline}
\endgroup
Moreover, up to some subsequence,
\begin{eqnarray} \label{ref:convergences}
\begin{cases}
\bullet \ p^0_k \rightarrow \tilde p^0 \\
\bullet \ (a_k,b_k) \rightarrow (\tilde a,\tilde b) \ \textnormal{, for the strong topology of} \ C^0 \\
\bullet \ (\zeta_k,\eta_k) \rightarrow (\tilde \zeta,\tilde \eta) \ \textnormal{, for the strong topology of} \ C^0 \\
\bullet \ (v_k,w_k) \rightarrow (\tilde v,\tilde w) \ \textnormal{, for the weak topology of} \ L^2 .
\end{cases}
\end{eqnarray}

\begin{remark}
Let $e : M \rightarrow \mathbb{R}^N$ be the canonical immersion. By definition, $(G^1)^{-1}(0)$, $(G^2)^{-1}(0)$ are submanifolds of $\mathbb{R}^N$ of dimension $N-r_1$, $N-r_2$, respectively, while $(g^1)^{-1}(0) = e^{-1}((G^1)^{-1}(0))$, $(g^2)^{-1}(0) = e^{-1}((G^2)^{-1}(0))$ are submanifolds of $M$ of dimension $n-r_1$, $n-r_2$, respectively. In particular,
\begin{equation} \label{ref:consistencyCond}
(G^1)^{-1}(0) \cap M = g^{-1}_1(0) \ , \ (G^2)^{-1}(0) \cap M = g^{-1}_2(0) .
\end{equation}
By transversality arguments\footnote{\begingroup \footnotesize As a matter of fact, let $\varphi$ be a chart of $x_i$ in $\mathbb{R}^N$, adapted to $M$. In these coordinates, one has $d_{x_i}G^i(\partial^j_{x_i}) = d_{x_i}g_i(\partial^j_{x_i})$, for $j=1,\dots,n$. If $v = \sum^n_{j=1} v_j \partial^j_{x_i} \in T_{x_i}M$ such that $d_{x_i}G^i(v) = 0$, we have $d_{x_i}g_i(v) = \sum^n_{j=1} v_j d_{x_i}g_i(\partial^j_{x_i}) = \sum^n_{j=1} v_j d_{x_i}G^i(\partial^j_{x_i}) = d_{x_i}G^i(v) = 0$. The conclusion then follows. \endgroup}, for $x_1 \in g^{-1}_1(0)$, $x_2 \in g^{-1}_2(0)$,
\begingroup
\begin{equation} \label{ref:intersectionTangentSpace}
T_{x_i}(G^i)^{-1}(0) \cap T_{x_i}M = T_{x_i}g^{-1}_i(0) \quad , \quad i =1,2
\end{equation}
\endgroup
which, thanks to Grassmann formula, gives
\begingroup
\begin{equation} \label{ref:sumTangentSpace}
\mathbb{R}^N = T_{x_i}(G^i)^{-1}(0) + T_{x_i}M \quad , \quad i=1,2 .
\end{equation}
\endgroup
\end{remark}

The next step consists of taking advantage of relations \eqref{ref:adjointRLEOCP}-\eqref{ref:transversalityRLEOCP} to recover an extremal for problem (\textbf{EOCP}). This goes by properly gluing the coordinates of the tuple $(\tilde p^0,\tilde a,\tilde b,\tilde \zeta,\tilde \eta,\tilde v,\tilde w)$ as follows. For every $t \in [0,t_2]$, define
$$
\tilde x(t) = \bigg\{ \begin{array}{cc}
\tilde a(t / t_1) & 0 \le t < t_1 \\
\tilde b((t - t_1) / (t_2 - t_1)) & t_1 \le t \le t_2
\end{array}
$$
$$
\tilde \gamma(t) = \bigg\{ \begin{array}{cc}
\tilde \zeta(t / t_1) & 0 \le t < t_1 \\
\tilde \eta((t - t_1) / (t_2 - t_1)) & t_1 \le t \le t_2 .
\end{array}
$$
$$
\tilde u(t) = \bigg\{ \begin{array}{cc}
\tilde v(t / t_1) & 0 \le t < t_1 \\
\tilde w((t - t_1) / (t_2 - t_1)) & t_1 \le t \le t_2
\end{array}
$$
Thanks to the structure of (\textbf{REOCP}), one sees that $(\tilde x,\tilde u)$ is feasible for (\textbf{EOCP}). Therefore, the tuple $(\tilde p^0,\tilde \lambda = (\tilde x,\tilde \gamma),\tilde u)$, where $(\tilde p^0,\tilde \gamma)$ is not trivial, is an extremal for problem (\textbf{EOCP}), i.e., it satisfies Theorem \ref{ref:theo_PMPGEO} in the case $M = \mathbb{R}^N$. Indeed, Remark \ref{remarkAdjoint} ensures that \eqref{ref:adjointRLEOCP}-\eqref{ref:adjointRLEOCP1} are equivalent to the adjoint equations \eqref{ref:adjointEqGEO} in the variable $\tilde \lambda$ in $[0,t_2]$, while, the independence of variables $(u,v) \in U^2$ and \eqref{ref:maxRLEOCP} leads to the maximality condition \eqref{ref:maxCondGEO}. Finally, since $g^{-1}_i(0) \subseteq (G^i)^{-1}(0)$, relations \eqref{ref:transversalityRLEOCP}, \eqref{ref:intersectionTangentSpace} gives the transversality condition \eqref{ref:trasv1GEO}. \\

By defining
$$
\gamma_k(t) = \bigg\{ \begin{array}{cc}
\zeta_k(t / t_1) & 0 \le t < t_1 \\
\eta_k((t - t_1) / (t_2 - t_1)) & t_1 \le t \le t_2 ,
\end{array}
$$
we remark that relations \eqref{ref:adjointRLEOCPk}-\eqref{ref:transversalityRLEOCPk} show that the tuple $(p^0_k,\lambda_k = (x_k,\gamma_k),u_k)$, where $(p^0_k,\gamma_k)$ is not trivial, represents an extremal for problem (\textbf{LEOCP})$_k$ and that \eqref{ref:convergences} leads to
\begin{eqnarray*}
\begin{cases}
\bullet \ x_k \rightarrow \tilde x \ \textnormal{, for the strong topology of} \ C^0 \\
\bullet \ \gamma_k|_{[0,t_1)} \rightarrow \tilde \gamma|_{[0,t_1)} \ \textnormal{, for the strong topology of} \ C^0 \\
\bullet \ \gamma_k|_{[t_1,t_2]} \rightarrow \tilde \gamma|_{[t_1,t_2]} \ \textnormal{, for the strong topology of} \ C^0 \\
\bullet \ u_k \rightarrow \tilde u \ \textnormal{, for the weak topology of} \ L^2
\end{cases}
\end{eqnarray*}
giving the convergence statement provided in Theorem \ref{ref:theoSCP}.

\subsection{Projection onto a Geometric Extremal for (\textbf{OCP})}

The main objective of this section is to prove that the piecewise absolutely continuous curve $\tilde \lambda$ can be conveniently projected onto $T^*M$, providing then the sought geometric extremal $(\tilde p^0,\tilde p,\tilde u)$ for the original problem (\textbf{OCP}). \\

We first introduce a useful projection operator. Consider the usual cotangent bundles $T^*M$ and $T^*\mathbb{R}^N$, and define
$$
T^* \mathbb{R}^N|_M = \bigcup_{q \in M} \ \{ q \} \times T^*_q \mathbb{R}^N .
$$
Equipped with the structure of pullback bundle, the projection $\pi_M : T^*\mathbb{R}^N|_M \rightarrow M$ is a vector bundle of rank $N$. In particular, $T^*M$ is identified to a subbundle of $T^*\mathbb{R}^N|_M$. We build a projection operator from cotangent spaces by considering the usual orthogonal projection. For this, let $q \in M$ and $(V,\varphi)$ be a local chart of $q$ in $\mathbb{R}^N$ adapted to $M$, i.e., $\varphi(V \cap M) = \varphi(V) \cap \mathbb{R}^{n} \times \{ 0 \}^{N-n}$. Then, $\{ dx^j(\cdot) \}_{j=1,\dots,N}$ is a local basis for $T^*\mathbb{R}^N|_M$ and $\{ dx^j(\cdot) \}_{j=1,\dots,n}$ is a local basis for $T^*M$ around $q$. Consider the cometric $\langle \cdot , \cdot \rangle(\cdot)$ in $T^*\mathbb{R}^N|_M$ induced by the Euclidean scalar product of $\mathbb{R}^N$. The Gram-Schmidt process on $\{ dx^j(\cdot) \}_{j=1,\dots,N}$ provides an orthonormal frame $\{ E_j(\cdot) \}_{j=1,\dots,N}$ for $T^*\mathbb{R}^N|_M$ in $V$, satisfying
\begingroup
\begin{equation} \label{ref:Span}
\textnormal{span} \langle E_1(\cdot) , \dots , E_j(\cdot) \rangle|_{V \cap M} = \textnormal{span} \langle dx^1(\cdot) , \dots , dx^j(\cdot) \rangle
\end{equation}
\endgroup
for every $1 \le j \le n$. From the previous results, when restricted to $V \cap M$, the orthogonal projection operator
\begingroup
\begin{equation} \label{ref:Project}
\begin{split}
\textnormal{Pr} : \ &T^* \mathbb{R}^{N+1}|_{(\mathbb{R} \times M)} \rightarrow T^*(\mathbb{R} \times M) \\
&(p^0,x,\gamma) \mapsto \bigg(p^0,\sum_{j=1}^{n} \langle \gamma , E_j(x) \rangle_x \ E_j(x) \bigg)
\end{split}
\end{equation}
\endgroup
is well-defined and smooth. Finally, since the change of frame mapping between two orthonormal frames is orthogonal, thanks to \eqref{ref:Span}, it is easy to see that the expression \eqref{ref:Project} is globally well-defined and smooth. \\

At this step, recall that any admissible trajectory for (\textbf{EOCP}) is also admissible for (\textbf{OCP}) (see Lemma \ref{ref:Microlemma}). Therefore, it holds $\tilde x([0,t_2]) \subseteq M$. By using the previous projection operator, this allows one to define the following piecewise absolutely continuous covector curve
$$
\tilde p : [0,t_2] \rightarrow T^*M : t \mapsto \pi_2 \circ \textnormal{Pr}(\tilde p^0,\tilde \lambda)
$$
where $\pi_2 : T^*(\mathbb{R} \times M) \rightarrow T^*M : (p^0,p) \mapsto p$. As a standard result for Hamiltonian systems, the curve $\tilde p$ satisfies the adjoint equations \eqref{ref:adjointEqGEO} in $M$ related to the Hamiltonian \eqref{ref:AugmentedHamiltonian}. This derives by combining Remark \ref{remarkAdjoint} (compare with the local system \eqref{ref:adjointDarboux}) with the following lemma, whose proof is reported here for sake of completeness (see \cite{Bonalli2018} for a more general statement).

\begin{lemma}
For almost every $t \in [0,t_2]$, choose a local chart $(U,\phi)$ of $(\tilde x^0,\tilde x)(t)$ in $\mathbb{R}^{N+1}$ (see system \eqref{ref:AugmentedAffineDynamics}) adapted to $\mathbb{R} \times M$. For every $i=0,\dots,n$, it holds
\begingroup
\begin{multline*}
\frac{d}{dt} \bigg( \textnormal{Pr} \Big( (\exp_{(F^0_{\tilde u},F_{\tilde u})}(t_2;t,\cdot))^*_{(\tilde x^0,\tilde x)(t_2)} \cdot \\
(\tilde p^0,\tilde \lambda(t_2)) \Big) \bigg( \frac{\partial}{\partial x^i} \bigg|_{(\tilde x^0,\tilde x)(t)} \bigg) \bigg)(t) = \\
-\sum_{j=0}^{n} \frac{\partial (f^0,f)_j}{\partial x^i}(\tilde x(t),\tilde u(t)) \bigg( \textnormal{Pr} \Big( (\exp_{(F^0_{\tilde u},F_{\tilde u})}(t_2;t,\cdot))^*_{(\tilde x^0,\tilde x)(t_2)} \cdot \\
(\tilde p^0,\tilde \lambda(t_2)) \Big)
\bigg( \frac{\partial}{\partial x^j} \bigg|_{(\tilde x^0,\tilde x)(t)} \bigg) \bigg)
\end{multline*}
\endgroup
where $(\cdot)^*$ is the pullback operator for 1-forms.
\end{lemma}

\addtolength{\textheight}{-8cm}   

\begin{proof}
For every $i=0,\dots,n$, denote $a_i(t) = \textnormal{Pr} \big( (\exp_{(F^0_{\tilde u},F_{\tilde u})}(t_2;t,\cdot))^*_{(\tilde x^0,\tilde x)(t_2)} \cdot (\tilde p^0,\tilde \lambda(t_2)) \big) \big( \frac{\partial}{\partial x^i} \big|_{(\tilde x^0,\tilde x)(t)} \big)$ and let $(V,\varphi)$ be an adapted local chart of $(\tilde x^0,\tilde x)(t_2)$. Since, for appropriate coefficients $b_i(t)$, the following relation holds
\begin{multline} \label{ref:linearCombinationAdjoint}
(\exp_{(F^0_{\tilde u},F_{\tilde u})}(t_2;t,\cdot))^*_{(\tilde x^0,\tilde x)(t_2)} \cdot (\tilde p^0,\tilde \lambda(t_2)) = \\ \sum_{j=0}^{N} b_j(t) dx^j|_{(\tilde x^0,\tilde x)(t)} \in T^* \mathbb{R}^{N+1}|_{\mathbb{R} \times M}
\end{multline}
from the span relation \eqref{ref:Span}, we obtain
\begin{multline*}
\textnormal{Pr} \big( (\exp_{(F^0_{\tilde u},F_{\tilde u})}(t_2;t,\cdot))^*_{(\tilde x^0,\tilde x)(t_2)} \cdot (\tilde p^0,\tilde \lambda(t_2)) \big) = \\
\sum_{j=0}^{n} b_j(t) dx^j|_{(\tilde x^0,\tilde x)(t)} \quad \scalebox{1.25}[1.25]{$\Longrightarrow$} \quad a_j(t) = b_j(t) \ , \ j=0,\dots,n .
\end{multline*}
By inverting expression \eqref{ref:linearCombinationAdjoint}, one then has
\begin{multline*}
(\tilde p^0,\tilde \lambda(t_2)) = \\
\sum_{j=0}^{n} a_j(t) (\exp_{(F^0_{\tilde u},F_{\tilde u})}(t;t_2,\cdot))^*_{(\tilde x^0,\tilde x)(t)} \cdot dx^j|_{(\tilde x^0,\tilde x)(t)} \\
+ \sum_{j=n+1}^{N} b_j(t) (\exp_{(F^0_{\tilde u},F_{\tilde u})}(t;t_2,\cdot))^*_{(\tilde x^0,\tilde x)(t)} \cdot dx^j|_{(\tilde x^0,\tilde x)(t)} .
\end{multline*}
However, since the trajectory $(\tilde x^0,\tilde x)(t) = \exp_{(F^0_{\tilde u},F_{\tilde u})}(t;t_2,\tilde x(t_2))$ lies entirely in $\mathbb{R} \times M$ and the chart $(U,\phi)$ is adapted to $\mathbb{R} \times M$, one sees that, for every $i=0,\dots,n$ and $j \ge n+1$, it holds
\begin{equation} \label{ref:nullComponents}
\begin{split}
&(\exp_{(F^0_{\tilde u},F_{\tilde u})}(t;t_2,\cdot))^*_{(\tilde x^0,\tilde x)(t)} \cdot dx^j|_{(\tilde x^0,\tilde x)(t)} \bigg( \frac{\partial}{\partial x^i} \bigg|_{(\tilde x^0,\tilde x)(t_2)} \bigg) \\
&= \frac{\partial}{\partial x^i} (x^j \circ \exp_{(F^0_{\tilde u},F_{\tilde u})}(t;t_2,\cdot) \circ \varphi^{-1})(\varphi((\tilde x^0,\tilde x)(t_2))) = 0 .
\end{split}
\end{equation}
Therefore, for every $i=0,\dots,n$, we have
\begin{multline*}
(\tilde p^0,\tilde \lambda(t_2)) \bigg( \frac{\partial}{\partial x^i} \bigg|_{(\tilde x^0,\tilde x)(t_2)} \bigg) = \\
\sum_{j=0}^{n} a_j(t) \frac{\partial}{\partial x^i} (x^j \circ \exp_{(F^0_{\tilde u},F_{\tilde u})}(t;t_2,\cdot) \circ \varphi^{-1})(\varphi((\tilde x^0,\tilde x)(t_2))) .
\end{multline*}
The term on the left of the previous expression does not depend on $t$, and then, by differentiating it w.r.t. $t$\footnote{Remark that, as soon as $0 \leqslant i \leqslant n$, quantities in \eqref{ref:eqFinalLemma} evolve in $M$. Therefore, indices grater than $n$ do not explicitly appear in computations.},
\begingroup
\begin{multline} \label{ref:eqFinalLemma}
\sum_{j=0}^{n} \bigg[ \dot{a}_j(t) \bigg( (\exp_{(F^0_{\tilde u},F_{\tilde u})}(t;t_2,\cdot))^*_{(\tilde x^0,\tilde x)(t)} \cdot \\
dx^j|_{(\tilde x^0,\tilde x)(t)} \bigg( \frac{\partial}{\partial x^i} \bigg|_{(\tilde x^0,\tilde x)(t_2)} \bigg) \bigg) + \\
\sum_{\ell=0}^{n} a_j(t) \frac{\partial (F^0,F)_j}{\partial x^{\ell}}(\tilde x(t),\tilde u(t)) \bigg( (\exp_{(F^0_{\tilde u},F_{\tilde u})}(t;t_2,\cdot))^*_{(\tilde x^0,\tilde x)(t)} \cdot \\
dx^{\ell}|_{(\tilde x^0,\tilde x)(t)} \bigg( \frac{\partial}{\partial x^i} \bigg|_{(\tilde x^0,\tilde x)(t_2)} \bigg) \bigg) \bigg] = 0 \ .
\end{multline}
and this must hold for every $i=0,\dots,n$. At this step, for the evaluation of $\displaystyle \frac{\partial (F^0,F)_j}{\partial x^{\ell}}(\tilde x(t),\tilde u(t))$, we notice that, in the case $\ell=0,\dots,n$, any variation point $\phi^{-1}(y_0,\dots,y_{\ell}+h,\dots,y_N)$, where $\phi^{-1}(y_0,\dots,y_{\ell},\dots,y_N) \in \mathbb{R} \times M$, belongs to $\mathbb{R} \times M$ because $\phi$ is adapted to $\mathbb{R} \times M$, and then, from the construction of $(F^0,F)$, we obtain $\displaystyle \frac{\partial (F^0,F)_j}{\partial x^{\ell}}(\tilde x(t),\tilde u(t)) = \displaystyle \frac{\partial (f^0,f)_j}{\partial x^{\ell}}(\tilde x(t),\tilde u(t))$. Moreover, by using \eqref{ref:nullComponents}, one has
\begingroup
\begin{multline*}
\hspace{-10pt} (\exp_{(F^0_{\tilde u},F_{\tilde u})}(t;t_2,\cdot))^*_{(\tilde x^0,\tilde x)(t)} \cdot \textnormal{span}\Big(dx^{n+1}|_{(\tilde x^0,\tilde x)(t)},\dots,dx^N|_{(\tilde x^0,\tilde x)(t)}\Big) \\
\subseteq \textnormal{span}\Big(dx^{n+1}|_{(\tilde x^0,\tilde x)(t_2)},\dots,dx^N|_{(\tilde x^0,\tilde x)(t_2)}\Big)
\end{multline*}
\endgroup
and, since $(\exp_{(F^0_{\tilde u},F_{\tilde u})}(t;t_2,\cdot))^*_{(\tilde x^0,\tilde x)(t)}$ is an isomorphism,
\begingroup
\begin{multline} \label{ref:spanExp}
\hspace{-10pt} (\exp_{(F^0_{\tilde u},F_{\tilde u})}(t;t_2,\cdot))^*_{(\tilde x^0,\tilde x)(t)} \cdot \textnormal{span}\Big(dx^0|_{(\tilde x^0,\tilde x)(t)},\dots,dx^n|_{(\tilde x^0,\tilde x)(t)}\Big) \\
\subseteq \textnormal{span}\Big(dx^0|_{(\tilde x^0,\tilde x)(t_2)},\dots,dx^n|_{(\tilde x^0,\tilde x)(t_2)}\Big) .
\end{multline}
\endgroup
Combining \eqref{ref:spanExp} with \endgroup\eqref{ref:eqFinalLemma} gives
\begingroup
\begin{multline*}
\sum_{j=0}^{n} \bigg[ \dot{a}_j(t) dx^j|_{(\tilde x^0,\tilde x)(t)} \\
+ \sum_{\ell=0}^{n} a_j(t) \frac{\partial (f^0,f)_j}{\partial x^{\ell}}(\tilde x(t),\tilde u(t)) dx^{\ell}|_{(\tilde x^0,\tilde x)(t)} \bigg] = 0 .
\end{multline*}
\endgroup
and the desired result follows straightforwardly.
\end{proof}

\vspace{10pt}

As last steps, we need to verify the maximality condition \eqref{ref:maxCondGEO}, the transversality conditions \eqref{ref:trasv1GEO} and the fact that the couple $(\tilde p^0,\tilde p)$ is not trivial. The first two facts follow by simple computations on relations \eqref{ref:maxRLEOCP}, \eqref{ref:transversalityRLEOCP}, respectively\footnote{\begingroup \footnotesize As a matter of fact, the transversality condition on the adjoint vector \eqref{ref:trasv1GEO} can be checked to hold for $\tilde p$ as follows. For every orthonormal frame $\{ E_1,\dots,E_N \}$ in $T^*_{\tilde x(t_i)}\mathbb{R}^N$ such that $\{ E_1,\dots,E_n \}$ spans $T^*_{\tilde x(t_i)}M$, the dual frames $\{ E^*_1,\dots,E^*_N \}$ and $\{ E^*_1,\dots,E^*_n \}$ span $T^{**}_{\tilde x(t_i)}\mathbb{R}^N \cong T_{\tilde x(t_i)}\mathbb{R}^N$ and $T^{**}_{\tilde x(t_i)}M \cong T_{\tilde x(t_i)}M$, respectively. Then, for $v \in T_{\tilde x(t_i)}g^{-1}_i(0) \subseteq T_{\tilde x(t_i)}M$ and $\lambda \in T^*_{\tilde x(t_i)}\mathbb{R}^N$, using these orthonormal basis and relations \eqref{ref:transversalityRLEOCP} one verifies straightforwardly that $\lambda(v) = (\pi_2\circ\textnormal{Pr})(\tilde p^0,\lambda)(v)$, from which the sought result follows. The previous equality is also used in what follows. \endgroup}.

By contradiction, assume that $(\tilde p^0,\tilde p) = 0$. In particular, $\tilde p(t^-_1) - \tilde p(t^+_1) = 0$ and $\tilde p(t_2) = 0$. On the other hand, the transversality conditions \eqref{ref:transversalityRLEOCP} read $\tilde \lambda(t^-_1) - \tilde \lambda(t^+_1) \perp T_{\tilde x(t_1)} (G^1)^{-1}(0)$ and $\tilde \lambda(t_2) \perp T_{\tilde x(t_2)} (G^2)^{-1}(0)$. Therefore, relations \eqref{ref:sumTangentSpace} provide that $\tilde \lambda(t^-_1) - \tilde \lambda(t^+_1) = 0$ and $\tilde \lambda(t_2) = 0$, which imply that $\tilde \gamma : [0,t_2] \rightarrow \mathbb{R}^N$ is a continuous curve which is zero at $t_2$ and satisfies a linear ODE. By uniqueness, we have $(\tilde p^0,\tilde \gamma) = 0$, which is a contradiction. \\

Theorem \ref{ref:theoSCP} is proved.}{}

\end{document}